\RequirePackage{fix-cm}
\documentclass[smallcondensed,envcountsame]{svjour3}     
\smartqed  
\usepackage[T1]{fontenc}
\usepackage[latin9]{inputenc}
\usepackage{geometry}
\usepackage{mathtools}
\usepackage{amstext}
\usepackage{amsmath}
\allowdisplaybreaks[1]
\usepackage{amssymb}
\usepackage{stmaryrd}
\usepackage{esint}
\makeatletter
\numberwithin{equation}{section}
\numberwithin{theorem}{section}
\numberwithin{figure}{section}
\spnewtheorem{assumption}{Assumption}{\bfseries}{\rmfamily}
\spnewtheorem*{lem*}{Lemma}{\itshape}{\rmfamily}
\spnewtheorem*{rem*}{Remark}{\itshape}{\rmfamily}

\usepackage{times}
\usepackage[colorlinks,citecolor=blue,urlcolor=blue]{hyperref}

\newcommand{\hd}{\delta}

\newcommand{\dis}{\varrho}

\makeatother

\begin{document}

\title{Schauder-type estimates for higher-order parabolic SPDEs\thanks{K. Du was partially supported by the National Natural Science Foundation of China (No. 11801084).}
}

\titlerunning{Higher-order SPDEs}        

\author{Yuxing Wang         \and
        Kai Du 
}

\authorrunning{Y.~Wang, K.~Du} 

\institute{Yuxing Wang \at
              School of Mathematical Sciences, Fudan University, Shanghai 200433, China \\
              \email{wangyuxing@fudan.edu.cn}           
           \and
           Kai Du (corresponding author)\at
              Shanghai Center for Mathematical Sciences, Fudan University, Shanghai 200438, China \\
              \email{kdu@fudan.edu.cn}
}

\date{}

\maketitle

\begin{abstract}
In this paper we consider the Cauchy problem for $2m$-order
stochastic partial differential equations of parabolic type 
in a class of stochastic H\"{o}lder spaces. 
The H\"{o}lder estimates of solutions and their spatial derivatives up
to order $2m$ are obtained, based on which the existence and 
uniqueness of solution is proved.
An interesting finding of this paper is that 
the regularity of solutions relies on a coercivity 
condition that differs when $m$ is odd or even: 
the condition for odd $m$ 
coincides with the standard parabolicity condition in the literature for higher-order stochastic partial differential equations,
while for even $m$ it depends on the integrability index $p$. 
The sharpness of the new-found coercivity condition is demonstrated by an example.

\end{abstract}

\section{Introduction}

Let $(\varOmega,\mathcal{F},(\mathcal{F}_{t})_{t\geq0},\mathbb{P})$ be a complete filtered probability space 
and $\{w_{\cdot}^{k}\}$ a sequence of independent standard Wiener processes 
adapted to the filtration $\mathcal{F}_{t}$.
Consider the Cauchy problem for the following $2m$-order stochastic partial differential
equations (SPDEs) of non-divergence form:
\begin{equation}
\text{\ensuremath{\mathrm{d}}}u=\bigg[(-1)^{m+1}\!\!\!\sum_{|\alpha|,\left|\beta\right|\leq m}A_{\alpha\beta}D^{\alpha+\beta}u+f\bigg]\mathrm{d}t+\sum_{k=1}^{\infty}\bigg[\sum_{|\alpha|\leq m}B_{\alpha}^{k}D^{\alpha}u+g^{k}\bigg]\mathrm{d}w_{t}^{k},\label{eq:dq}
\end{equation}
where the coefficients, the free terms, and the unknown function are all
random fields defined on $\mathbb{R}^{n}\times [0,\infty)\times\varOmega$
and adapted to $\mathcal{F}_{t}$.
Typical examples of Equation
(\ref{eq:dq}) include the
Zakai equation (see \cite{zakai1969optimal,rozovskii1990stochastic} for example),
linearised stochastic
Cahn--Hilliard equations
(see \cite{da1996stochastic,cardon2001cahn} for example), and so on.
General solvability theory for higher-order SPDEs of type (\ref{eq:dq}) 
was first investigated in \cite{krylov1979stochastic}
under the framework of Hilbert spaces.
This paper concerns the existence, uniqueness and regularity of solutions of (\ref{eq:dq})
in some H\"{o}lder-type spaces that will be defined later.
Regularity theory for linear equations often plays
an important role in the study of nonlinear stochastic equations, see
\cite{walsh1986introduction,chow2014stochastic,da2014stochastic}
and references therein.

The weak solution of Equation (\ref{eq:dq}), which satisfies the equation 
in the (analytic) distribution
sense, and its regularity in the framework of Sobolev spaces
have been investigated by many researchers. 
Results for the second-order case (namely $m=1$) are numerous and fruitful; for instance, a complete
$L^{p}$-theory $(p\geq2)$ of second-order parabolic SPDEs has been
developed, see for example \cite{pardoux1975equations,krylov1977cauchy,krylov1979stochastic,rozovskii1990stochastic}
for $p=2$ and \cite{krylov1996l_p,krylov1999analytic,krylov2000spdes,krylov200516,krylov2008brief}
for $p\geq2$; degenerate equations were
addressed in, for example, \cite{krylov1986characteristics}; 
and the Dirichlet problem were also extensively studied in many publications such as \cite{krylov1999sobolev,kim2004p,kim2004stochastic,cioica2013lq,kim2014weighted,lindner2014singular,kim2015bmo,du2018w2}.
For higher-order SPDEs, Krylov and Rozovskii \cite{krylov1979stochastic} applied their abstract result to obtain the existence and uniqueness of
solutions in the Sobolev space $W^{m}_2(\mathbb{R}^{n})$.
Recently,  
van Neeven et al.~\cite{van2012maximal} and 
Portal and Veraar \cite{portal2019stochastic}
obtained some maximal $L^{p}$-regularity results for strong solutions
of abstract stochastic parabolic time-dependent problems,
which can also apply to higher-order SPDEs with proper conditions. 

Another approach to the regularity problem of SPDEs is based on some H\"older spaces,
corresponding to the celebrated Schauder theory for classical elliptic and parabolic PDEs
(see \cite{gilbarg2015elliptic} and references therein).
This paper adopts this strategy to study Equation \eqref{eq:dq},
stimulated by recent progress of the related research on second-order SPDEs.
Actually, a $C^{2+\delta}$-theory for \eqref{eq:dq} with $m=1$ was once
an open problem proposed by Krylov~\cite{krylov1999analytic},
which was partially
addressed by Mikulevicius~\cite{mikulevicius2000cauchy}, and generally
solved by Du and Liu~\cite{du2019cauchy} very recently. 
Introducing a H\"older-type space $\mathcal{C}_{p}^{\delta}$ containing all random fields $u$ satisfying
\[
\|u\|_{\mathcal{C}_{p}^{\delta}}:=\bigg[\sup_{t,\,x}\mathbb{E}|u(x,t)|^{p}+\sup_{t,\,x\neq y}\frac{\mathbb{E}|u(x,t)-u(y,t)|^{p}}{|x-y|^{\delta p}}\bigg]^{\frac{1}{p}}<\infty
\]
with some constants $\delta \in(0,1)$ and $p\in[2,\infty)$,
they proved that, under natural conditions
on the coefficients, the solution $u$ and its derivatives $D u$
and $D^{2}u$ belong to $\mathcal{C}_{p}^{\delta}$,
provided that $f$, $g$ and $D g$ belong to the same space. 
In addition, Du and Liu~\cite{du2019cauchy} also obtained H\"{o}lder continuity in time of $D^{2}u$ with time irregular coefficients.
A similar $C^{2+\delta}$-theory was also obtained recently for systems of second-order SPDEs in \cite{du2017stochastic}.

This paper aims to prove a Schauder-type estimate for Equation (\ref{eq:dq})
based on the space $\mathcal{C}_{p}^{\delta}$.
To get more insight into such a kind of regularity of higher-order equations, 
let us recall some relevant work on deterministic PDEs.
Boccia~\cite{Boccia2013Schauder} derived Schauder estimates for solutions of $2m$-order parabolic systems of non-divergence form in the classical
$C_{x}^{2m+\hd}$-space provided that the free term $f$ (there are no terms like $g^k$
in deterministic equations) belongs to $C_{x}^{\hd}$,
and for the divergence form Dong and Zhang~\cite{dong2015schauder} 
obtained $C_{x}^{m+\hd}$ regularity.
Considering the feature of stochastic integral terms in SPDEs, a natural form of Schauder estimates for Equation (\ref{eq:dq}) must be like this:
\emph{the $\mathcal{C}_{p}^{\delta}$-norms of $u$ and its derivatives up to order $2m$
are dominated by the $\mathcal{C}_{p}^{\delta}$-norms of $f$ and $D^\alpha g$ with $|\alpha|\le m$}.

What surprises us during this work is 
not the above natural assertion but the structural condition
that ensures the validity of this assertion.
Let us give some explanation.
It is well-known that the classical Schauder estimate for PDEs or PDE systems is based on
certain coercivity conditions imposed on the leading coefficients and usually called
strong ellipticity or strong parabolicity,
and for second-order SPDEs either $L^p$-theory or $C^{2+\delta}$-theory
requires a stochastic version of such conditions
(see \cite{krylov1999analytic,du2019cauchy} for example).
The solvability result of higher-order SPDEs in the space $W^m_2(\mathbb{R}^n)$
obtained \cite{krylov1979stochastic} relied on the following condition:
there is a constant $\lambda>0$ such that for all $\xi_{\alpha}\in\mathbb{R}$,
\begin{align}\label{eq:parbolic1}
\sum_{|\alpha|=|\beta|=m}2A_{\alpha\beta}\xi_{\alpha}\xi_{\beta} 
- \lambda\sum_{|\alpha|=m}\left|\xi_{\alpha}\right|^{2}& \geq
\sum_{k=1}^{\infty}\bigg|\sum_{|\alpha|=m}B_{\alpha}^{k}\xi_{\alpha}\bigg|^{2}.
\end{align}
This is a natural condition as it can reduce to the standard ones 
for PDEs and for second-order SPDEs.
However, things may change when one considers $L^p$-integrability ($p>2$)
rather than square-integrability;
more specifically, the coercivity condition \eqref{eq:parbolic1}
being adequate for $L^2$-theory
seems not to be sufficient for $L^p$-integrability of solutions or their derivatives
when $m \ge 2$.
An indirect evidence is that, when the abstract maximal $L^{p}$-regularity results
obtained in \cite{van2012maximal,portal2019stochastic} applied to 
higher-order SPDEs of type (\ref{eq:dq})
the coefficients $B_\alpha$ with $|\alpha|=m$ were required to 
either be sufficiently small or have some additional analytic properties 
(see~\cite{portal2019stochastic} for details).
Similar phenomena have been found also in complex valued SPDEs (see~\cite{brzezniak2012stochastic})
and systems of second-order SPDEs (see~\cite{kim2013note,du2017stochastic}).
This seems to be a unique feature of stochastic equations in contrast to
deterministic PDEs.

A major contribution of this paper is the finding of a $p$-dependent
coercivity condition that is just a small modification of \eqref{eq:parbolic1}
but perfectly works for the Schauder
theory for Equation (\ref{eq:dq}) based on $\mathcal{C}_{p}^{\delta}$.
Let us state this condition as below: 
with some constants $\lambda>0$ and $p\ge 2$ it holds that
\begin{align}\label{eq:parbolic}
\sum_{|\alpha|=|\beta|=m}2A_{\alpha\beta}\xi_{\alpha}\xi_{\beta} 
- \lambda\sum_{|\alpha|=m}\left|\xi_{\alpha}\right|^{2}& \geq\frac{p+(-1)^{m}(p-2)}{2}\sum_{k=1}^{\infty}\bigg|\sum_{|\alpha|=m}B_{\alpha}^{k}\xi_{\alpha}\bigg|^{2}\\
 & = \left\{
\begin{aligned}
&\sum_{k=1}^{\infty}\bigg|\sum_{|\alpha|=m}B_{\alpha}^{k}\xi_{\alpha}\bigg|^{2} & \text{when }m\text{ is odd},\\
&(p-1)\sum_{k=1}^{\infty}\bigg|\sum_{|\alpha|=m}B_{\alpha}^{k}\xi_{\alpha}\bigg|^{2} & \text{when }m\text{ is even}.
\end{aligned}\right.
\nonumber
\end{align}
Obviously, this condition is really $p$-dependent \emph{only} when $m$ is even,
and for odd $m$ it turns to be the same with \eqref{eq:parbolic1}.
Though it might look strange at first glance, 
the following example demonstrates its sharpness to some extent.
\begin{example}
\label{exa:plambda} Given $\mu\in\mathbb{R}$, we consider the following equation on the torus $\mathbb{T}\coloneqq\mathbb{R}/(2\pi\mathbb{Z})$:
\begin{equation}\label{eq:eg}
\textrm{d}u=(-1)^{m+1}D^{2m}u\,\textrm{d}t+\mu D^{m}u\,\textrm{d}w_{t}
\end{equation}
with the initial condition 
$$u(x,0)=\sum_{n\in\mathbb{Z}}\mathrm{e}^{-n^{2m}}\cdot\mathrm{e}^{\sqrt{-1}nx},
\quad x\in\mathbb{T}.$$ 
If $\mu^{2}<2$,
from Theorem~3.2.1 in \cite{krylov1979stochastic} this equation admits a unique solution $u$ in $L^{2}(\Omega;C([0,T];H^{l}(\mathbb{T})))$
for any integer $l$.
However, we have the following lemma.
\end{example}

\begin{lemma}
\label{lem:sharpness}Let $m$ be even and $\mu^{2}<2$. If $p>1+2/\mu^2$,
then $\mathbb{E}\left\Vert u(\cdot,t)\right\Vert _{L^{2}(\mathbb{T})}^{p}=+\infty$
for any $t>2/\varepsilon$ with $\varepsilon=(p-1)\mu^{2}-2$. Consequently, $\sup_{x\in\mathbb{T}}\mathbb{E}|u(x,t)|^{p}=+\infty$
for any $t>2/\varepsilon$.
\end{lemma}

The proof of Lemma \ref{lem:sharpness} is presented in Section 6.
This result indicates that the coefficient $p-1$ in the even case of the condition \eqref{eq:parbolic}
couldn't get any smaller if one wants to always ensure the finiteness of $\sup_{x}\mathbb{E}|u(x,t)|^{p}$, and this, of course, is a basic requirement in our theory.

Although our main result, Theorem \ref{thm:2} below, is stated (and also proved) only for linear equations of form (\ref{eq:dq}), we point out that it is not difficult to extend it to the semilinear case where $f$ and $g$ depend on the unknown $u$ and are Lipschitz continuous with respect to all $D^\alpha u$ with $|\alpha|<2m$ and to all $D^\beta u$ with $|\beta|<m$, respectively. 
Besides, it is also interesting to ask if the coercivity condition \eqref{eq:parbolic} is
sufficient or not to construct an $L^p$-theory for Equation (1.1).

Our approach to Schauder estimates, following the strategy used in \cite{du2019cauchy,du2017stochastic}, combines a perturbation
scheme of Wang \cite{wang2006schauder} with some integral-type
estimates that were also used in \cite{trudinger1986new}. 
The effect of the $p$-dependent condition \eqref{eq:parbolic}
can be seen in the proof of the mixed norm estimates (Lemma~\ref{lem:the whole space estimate});
the latter leads to a local boundedness estimate that plays a key
role in proving the fundamental interior Schauder estimate for the model equation 
(see~\eqref{eq:model} below).

This paper is organised as follows. 
In the next section we state our main theorem
after introducing some
notation and assumptions. 
Sections 3 and 4
are both devoted to the estimates for the model equation 
whose coefficients depend on $t$ and $\omega$ but not on
$x$; we prove some auxiliary estimates in Section~3,
and establish
the interior H\"{o}lder estimate in Section~4.
The proof of the main theorem is completed in Section~5. 
In the final section we prove Lemma~\ref{lem:sharpness}.

\section{Notation and main results}

Before stating the main results, we introduce some notation
and the working spaces. For a function $f$ of $x=(x_{1},\ldots,x_{n})\in\mathbb{R}^{n}$
and a multi-index $\beta=(\beta_{1},\ldots,\beta_{n})\in\mathbb{N}^{n}$,
we define
\[
D^{\beta}f=\frac{\partial^{\beta_{1}}\cdots\partial^{\beta_{n}}}{\partial x_{1}^{\beta_{1}}\cdots\partial x_{n}^{\beta_{n}}}f,\;\left|\beta\right|=\beta_{1}+\cdots+\beta_{n}.
\]
For $k\in\mathbb{N}=\{0,1,2,\dots\}$, $D^{k}f$ is regarded as the set of all $k$-order
derivatives of $f$ and $\Vert D^{k}f\Vert_{E}={\sum_{|\alpha|=k}\Vert D^{\alpha}f\Vert_{E}}$
where $\Vert\cdot\Vert_{E}$ is the norm of a normed space $E$. All
the derivatives of $E$-valued functions are defined with respect
to the spatial variables in the strong sense as in \cite{phillips1957functional}.

A Banach space-valued H\"{o}lder continuous function
is a natural extension of the classical H\"{o}lder continuous function.
Let $E$ be a Banach space, $\mathcal{O}$ be a domain in $\mathbb{R}^{n}$,
$I\subset\mathbb{R}$ be an interval, and $Q\coloneqq\mathcal{O}\times I$.
For a function $h:\,\mathcal{O}\rightarrow E$, we define
\begin{eqnarray*}
 &  & \left|h\right|_{k;\mathcal{O}}^{E}\coloneqq\max_{\left|\beta\right|\leq k}\sup_{x\in\mathcal{O}}\left\Vert D^{\beta}h(x)\right\Vert _{E},\\
 &  & \left[h\right]_{k+\hd;\mathcal{O}}^{E}\coloneqq\max_{\left|\beta\right|=k}\sup_{x,y\in\mathcal{O},x\neq y}\frac{\left\Vert D^{\beta}h(x)-D^{\beta}h(y)\right\Vert _{E}}{\left|x-y\right|^{\hd}},\\
 &  & \left|h\right|_{k+\hd;\mathcal{O}}^{E}\coloneqq\left|h\right|_{k;\mathcal{O}}^{E}+\left[h\right]_{k+\hd;\mathcal{O}}^{E}.
\end{eqnarray*}
with $k\in\mathbb{N}$ and $\alpha\in(0,1)$. For a function
$u:\,Q\rightarrow E$, we define
\begin{eqnarray*}
\left[u\right]_{k+\hd;Q}^{E}\coloneqq\sup_{t\in I}\left[u(\cdotp,t)\right]_{k+\hd;\mathcal{O}}^{E}, &  & \left|u\right|_{k+\hd;Q}^{E}\coloneqq\sup_{t\in I}\left|u(\cdotp,t)\right|_{k+\hd;\mathcal{O}}^{E}.
\end{eqnarray*}
Moreover, we define the parabolic modulus $\left|X\right|_{\mathrm{p}}=\left|(x,t)\right|_{\mathrm{p}}\coloneqq\left|x\right|+\left|t\right|^{\frac{1}{2m}}$
and 
\begin{eqnarray*}
 &  & \left[u\right]_{(k+\hd,\hd/2m);Q}^{E}\coloneqq\max_{\left|\beta\right|=k}\sup_{X,Y\in Q,X\neq Y}\frac{\left\Vert D^{\beta}u(X)-D^{\beta}u(Y)\right\Vert _{E}}{\left|X-Y\right|_{\mathrm{p}}^{\hd}},\\
 &  & \left|u\right|_{(k+\hd,\hd/2m);Q}^{E}\coloneqq\left|u\right|_{k;Q}^{E}+\left[u\right]_{(k+\hd,\hd/2m);Q}^{E}.
\end{eqnarray*}
In this paper, $E$ is either i) $\mathbb{R}$, ii) $l^{2}$ or iii)
$L_{\omega}^{p}\coloneqq L^{p}(\Omega)$. We omit the superscript
in cases i) and ii), and in case iii) we denote
\[
\interleave\cdotp\interleave_{\ldots}\coloneqq\left|\cdotp\right|_{\ldots}^{L_{\omega}^{p}},\;\left\llbracket \cdotp\right\rrbracket _{\ldots}\coloneqq\left[\cdotp\right]_{\ldots}^{L_{\omega}^{p}}
\]
for simplicity.
\begin{definition}
The H\"{o}lder-type spaces $C_{x}^{k+\hd}(Q;L_{\omega}^{p})$ and
$C_{x,t}^{k+\hd,\hd/2m}(Q;L_{\omega}^{p})$ are defined as all predictable
random fields $u$ defined on $Q\times\Omega$ and taking values in
an Euclidean space or $l^{2}$ such that $u(\cdot,t)$ is an $L_{\omega}^{p}$-valued
strongly continuous function for each $t$, and $\interleave u\interleave_{k+\hd;Q}$
and $\interleave u\interleave_{(k+\hd,\hd/2m);Q}$ are finite respectively.
\end{definition}

Obviously, a function $u$ in $C_{x}^{k+\hd}(Q;L_{\omega}^{p})$ means that itself and
its spatial derivatives up to order $k$ lie in the space $\mathcal{C}_{p}^{\delta}$
defined in the previous section.

In this paper we adopt a concept of quasi-classical solutions introduced in \cite{du2019cauchy}.
\begin{definition}
A predictable random field \emph{u} is called a \emph{quasi-classical
}solution of (\ref{eq:dq}) if

(i) for each $t\in(0,\infty)$, $u(\cdot,t)$ is an $2m$ times strongly
differentiable function from $\mathbb{R}^{n}$ to $L_{\omega}^{p}$
for some $p\geq2$; and 

(ii) for each $x\in\mathbb{R}^{n}$, the process $u(x,\cdot)$ is
stochastically continuous and satisfies the integral equation
\begin{eqnarray}
u(x,T_{1})-u(x,T_{0})
 & = & \int_{T_{0}}^{T_{1}}\bigg[-(-1)^{m}\sum_{|\alpha|,\left|\beta\right|\leq m}A_{\alpha\beta}D^{\alpha+\beta}u(x,t)+f(x,t)\bigg]\mathrm{d}t\\
 &  & +\int_{T_{0}}^{T_{1}}\sum_{k=1}^{\infty}\bigg[\sum_{|\alpha|\leq m}B_{\alpha}^{k}D^{\alpha}u(x,t)+g^{k}(x,t)\bigg]\mathrm{d}w_{t}^{k}\nonumber 
\end{eqnarray}
almost surely (a.s.) for all $0\leq T_{0}<T_{1}<\infty$.
\end{definition}

In particular, if $u(\cdot,t,\omega)\in C^{2m}(\mathbb{R}^{n})$ for
each $(t,\omega)\in [0,\infty)\times\Omega$, then $u$ is a \emph{classical
}solution of (\ref{eq:dq}).

Next we will introduce some notations for the domains:
\[
B_{r}(x)\coloneqq\{y\in\mathbb{R}^{n}:|y-x|<r\},\:Q_{r}(x,t)\coloneqq B_{r}(x)\times(t-r^{2m},t]
\]
and simply write $B_{r}\coloneqq B_{r}(0),$ $Q_{r}\coloneqq Q_{r}(0,0)$.
Also we denote 
\[
\mathcal{Q}_{r,T}(x)\coloneqq B_{r}(x)\times(0,T],\:\mathcal{Q}_{r,T}\coloneqq\mathcal{Q}_{r,T}(0),\:\mathcal{Q}_{T}\coloneqq\mathbb{R}^{n}\times(0,T].
\]

\begin{assumption}\label{ass}
The following conditions hold throughout the paper unless otherwise
stated:
\begin{itemize}
\item[1)] The coercivity condition \eqref{eq:parbolic} is satisfied with some $\lambda > 0$ 
and $p\ge 2$.
\item[2)] The random fields $A_{\alpha\beta}$ and $f$ are real-valued, and $B_{\alpha}$ and
$g$ are $l^{2}$-valued; all of them are predictable. 
The classical
$C_{x}^{\hd}$-norms of $A_{\alpha\beta}(\cdot,t,\omega)$ and $C_{x}^{m+\hd}$-norms
of $B_{\alpha}(\cdot,t,\omega)$ are all dominated by a constant $K>0$ uniformly in $(t,\omega)$.
\item[3)] The free terms $f\in C_{x}^{\hd}(\mathcal{Q}_{T};L_{\omega}^{p})$ and
$g\in C_{x}^{m+\hd}(\mathcal{Q}_{T};L_{\omega}^{p})$.
\end{itemize}
\end{assumption}

Now we are ready to state the main result in this paper which consists
of the global H\"{o}lde estimate and the solvability.
\begin{theorem}
\emph{\label{thm:2}}Under Assumptions \ref{ass}, there
exists a unique quasi-classical solution $u \in C_{x,t}^{2m+\hd,\hd/2m}(\mathcal{Q}_{T};L_{\omega}^{p})$ to Equation~(\ref{eq:dq}) with the
initial condition $u(\cdot,0)=0$. Moreover, there is a constant
$C>0$ depending only on $n$, m, $\lambda$, p, $\hd$ and K such that
\emph{
\begin{equation}
\interleave u\interleave{}_{(2m+\hd,\hd/2m);\mathcal{Q}_{T}}\leq Ce^{CT}(\interleave f\interleave_{\hd;\mathcal{Q_{T}}}+\interleave g\interleave_{m+\hd;\mathcal{Q}_{T}}).\label{eq:global estimate}
\end{equation}
}
\end{theorem}

In the proof of Theorem \ref{thm:2} the global H\"{o}lde
estimate \eqref{eq:global estimate} is derived first, and then the existence and uniqueness of solutions of Equation (\ref{eq:dq}) is obtained by the standard method of continuity. 

We remark that the Cauchy problem with nonzero initial condition can be reduced into
the case of zero initial condition by some simple calculation. 
Also, if $p$ is large enough one can obtain a modification
of the solution that is H\"{o}lder continuous
jointly in space and time by means of the Kolmogorov
continuity theorem (see \cite{dalang2007hitting} for example).

\section{Auxillary estimates for the model equation}

In Sections 3 and 4 we always assume that 
the coefficients $A_{\alpha\beta}$ and $B_{\alpha}$ with $|\alpha|=|\beta|=m$ are all bounded predictable processes (dominated by the constant $K$),
independent of the spatial variable $x$, and satisfy the coercivity condition \eqref{eq:parbolic}.
Consider the following model equation
\begin{align}
\mathrm{d}u(x,t)= & \bigg[-(-1)^{m}\sum_{|\alpha|=\left|\beta\right|=m}A_{\alpha\beta}(t)D^{\alpha+\beta}u(x,t)+f(x,t)\bigg]\mathrm{d}t\nonumber\\
 & +\sum_{k=1}^\infty\bigg[\sum_{|\alpha|=m}B_{\alpha}^{k}(t)D^{\alpha}u(x,t)+g^{k}(x,t)\bigg]\mathrm{d}w_{t}^{k} \label{eq:model}
\end{align}
with $(x,t)\in\mathbb{R}^{n}\times[-1,+\infty).$

Let $\mathcal{O\subset\mathbb{R}}^{n}$ and $H^{k}(\mathcal{O})=W^{k,2}(\mathcal{O})$
be the usual Sobolev spaces. Let $I\subset\mathbb{R}$ and $Q=\mathcal{O}\times I$.
For $p,q\in[1,\infty]$,
define 
\[
L_{\omega}^{p}L_{t}^{q}H_{x}^{m}(Q):=L^{p}(\varOmega;L^{q}(I;H^{m}(\mathcal{O}))),
\]
and the domain $Q$ in the notation will be often omitted if there is no confusion.

\begin{lemma}
\label{lem:the whole space estimate}Let $\mathcal{Q}_{T}=\mathbb{R}^{n}\times[0,T]$,
$p\geq2$ and the integer $l\ge m$. Suppose $f\in L_{\omega}^{p}L_{t}^{2}H_{x}^{l-m}(\mathcal{Q}_{T})$
and $g\in L_{\omega}^{p}L_{t}^{2}H_{x}^{l}(\mathcal{Q}_{T})$. Then Equation (\ref{eq:model})
with zero initial value admits a unique weak solution $u\in L_{\omega}^{p}L_{t}^{\infty}H_{x}^{l}(\mathcal{Q}_{T})\cap L_{\omega}^{p}L_{t}^{2}H_{x}^{l+m}(\mathcal{Q}_{T})$.
Moreover, for any multi-index $\beta$ such that $|\beta|\leq l$,
\begin{equation}
\left\Vert D^{\beta}u\right\Vert {}_{L_{\omega}^{p}L_{t}^{\infty}L_{x}^{2}}+\left\Vert D^{\beta}D^{m}u\right\Vert {}_{L_{\omega}^{p}L_{t}^{2}L_{x}^{2}}\leq C(\left\Vert D^{\beta}f\right\Vert {}_{L_{\omega}^{p}L_{t}^{2}H_{x}^{-m}}+\left\Vert D^{\beta}g\right\Vert {}_{L_{\omega}^{p}L_{t}^{2}L_{x}^{2}})\label{eq:p estimate}
\end{equation}
where the constant $C$ depends only on $n$, $p$, $m$, $T$, $\lambda$, and $K$.
\end{lemma}

\begin{proof}
If $p=2$, the existence and uniqueness of the weak solution has already
been obtained in \cite{krylov1979stochastic} and \cite{rozovskii1990stochastic}.
So it remains to prove estimate (\ref{eq:p estimate}) for general
$p\geq2.$ Since we can differentiate (\ref{eq:model}) with order
$\beta$, it suffices to prove the estimate in the case $|\beta|=0$ .

By an It\^{o} formula from \cite[Theorem~1.3.1]{krylov1979stochastic},
one can derive 
\begin{eqnarray}
 &  & \mathrm{d}\Vert u(\cdot,t)\Vert_{L_{x}^{2}}^{2}\label{eq:2 order}\\
 & = & {\displaystyle \int_{\mathbb{R}^{n}}[-2\sum_{|\alpha|=|\beta|=m}A_{\alpha\beta}D^{\alpha}uD^{\beta}u+2uf+\sum_{k=1}^{\infty}|\sum_{|\alpha|=m}B_{\alpha}^{k}D^{\alpha}u+g^{k}|^{2}]\mathrm{d}x\mathrm{d}t}\nonumber \\
 &  & +\sum_{k=1}^{\infty}\int_{\mathbb{R}^{n}}2u(\sum_{|\alpha|=m}B_{\alpha}^{k}D^{\alpha}u+g^{k})\mathrm{d}x\mathrm{d}w_{t}^{k}.\nonumber 
\end{eqnarray}
Note that in the last term one has
\begin{equation}\label{eq:002}
\int_{\mathbb{R}^{n}}2u\Big(\sum_{|\alpha|=m}B_{\alpha}^{k}D^{\alpha}u\Big)\mathrm{d}x=0
\quad\text{when }m \text{ is odd,}
\end{equation}
but it is not true for even $m$. 

Take a stopping
time $\tau$ such that
\[
\mathbb{E}\left[\bigg(\sup_{t\in[0,\tau]}\Vert u(t)\Vert_{L_{x}^{2}}^{2}+\int_{0}^{\tau}\Vert D^{m}u(t)\Vert_{L_{x}^{2}}^{2}\mathrm{d}t\bigg)^{\frac{p}{2}}\right]<+\infty.
\]
Let us consider two cases:

{\bf Case 1.} $m$ is \emph{odd}. Using the fact \eqref{eq:002}, and by Condition (\ref{eq:parbolic}),
the Sobolev--Gagliargo--Nirenberg inequality and Young's inequality, we
have 
\begin{eqnarray*}
 &  & \sup_{t\in[0,\tau]}\left\Vert u(t)\right\Vert _{L_{x}^{2}}^{2}+\int_{0}^{\tau}\left\Vert D^{m}u(t)\right\Vert _{L_{x}^{2}}^{2}\mathrm{d}t\\
 & \leq & C\int_{0}^{\tau}\big[\left\Vert f(t)\right\Vert _{H_{x}^{-m}}^{2}+\left\Vert g(t)\right\Vert _{L_{x}^{2}}^{2}\big]\mathrm{d}t+C\left|\sup_{t\in[0,\tau]}\sum_k\int_{0}^{t}\int_{\mathbb{R}^{n}}2ug^{k}\mathrm{d}x\mathrm{d}w_{s}^{k}\right|.
\end{eqnarray*}
Then computing $\mathbb{E}[\cdot]^{p/2}$ on both sides of the above inequality, and using the Burkholder--Davis--Gundy (BDG)
inequality and Young's inequality, one can obtain that
\begin{eqnarray}\label{eq:001}
\mathbb{E}{\displaystyle \sup_{t\in[0,\tau]}\Vert u(t)\Vert_{L_{x}^{2}}^{p}}+\mathbb{E}\left(\int_{0}^{\tau}\Vert D^{m}u\Vert_{L_{x}^{2}}^{2}\mathrm{d}t\right)^{\frac{p}{2}}
\leq  C\mathbb{E}\left|\int_{0}^{\tau}(\Vert f\Vert_{H_{x}^{-m}}^{2}+\Vert g\Vert_{L_{x}^{2}}^{2})\mathrm{d}t\right|^{\frac{p}{2}},
\end{eqnarray}
where the constant $C$ depends only on $n$, $p$, $m$, $T$, and $\lambda$.

{\bf Case 2.} $m$ is \emph{even}. Applying
It\^{o}'s formula to $\Vert u(\cdot,t)\Vert_{L_{x}^{2}}^{p}$, one
can derive
\begin{eqnarray}
 &  & \mathrm{d}\Vert u(\cdot,t)\Vert_{L_{x}^{2}}^{p}\label{eq:p order}\\
 & = & \frac{p}{2}\Vert u\Vert_{L_{x}^{2}}^{p-2}\int_{\mathbb{R}^{n}} \bigg[\Big(-2\sum_{|\alpha|=|\beta|=m}A_{\alpha\beta}D^{\alpha}uD^{\beta}u+\sum_{k=1}^{+\infty}|\sum_{|\alpha|=m}B_{\alpha}^{k}D^{\alpha}u|^{2}\Big)\nonumber \\
 &  & +2uf+2\sum_{k}\sum_{|\alpha|=m}(B_{\alpha}^{k}D^{\alpha}u)g^{k}+|g|^{2} \bigg]\mathrm{d}x\mathrm{d}t\nonumber \\
 &  & +p\Vert u\Vert_{L_{x}^{2}}^{p-2}\sum_k\int_{\mathbb{R}^{n}}u\Big(\sum_{|\alpha|=m}B_{\alpha}^{k}D^{\alpha}u+g^{k}\Big)\mathrm{d}x\mathrm{d}w_{t}^{k}\nonumber \\
 &  & +\frac{p(p-2)}{2}\textbf{1}_{\{\Vert u\Vert_{L_{x}^{2}}\neq0\}}\Vert u\Vert_{L_{x}^{2}}^{p-4}\sum_{k}\bigg[\int_{\mathbb{R}^{n}}u\Big(\sum_{|\alpha|=m}B_{\alpha}^{k}D^{\alpha}u+g^{k}\Big)\mathrm{d}x\bigg]^{2} \mathrm{d}t.\nonumber 
\end{eqnarray}
With the help of H\"{o}lder inequality, one can obtain
\begin{eqnarray}
 &  & (p-2)\sum_{k}\bigg[\int_{\mathbb{R}^{n}}u\Big(\sum_{|\alpha|=m}B_{\alpha}^{k}D^{\alpha}u+g^{k}\Big)\mathrm{d}x\bigg]^{2} \label{eq:small}\\
 & \leq & (p-2)(1+\epsilon)\sum_{k}\bigg(\int_{\mathbb{R}^{n}}\sum_{|\alpha|=m}uB_{\alpha}^{k}D^{\alpha}u\mathrm{d}x\bigg)^{2}+C(\epsilon,p,m)\sum_{k}\bigg(\int_{\mathbb{R}^{n}}ug^{k}\mathrm{d}x\bigg)^{2}\nonumber \\
 & \leq & (1+\epsilon)(p-2)\Vert u\Vert_{L_{x}^{2}}^{2}\text{\ensuremath{\cdot}}\sum_{k}\int_{\mathbb{R}^{n}}\bigg|\sum_{|\alpha|=m}B_{\alpha}^{k}D^{\alpha}u\bigg|^{2}\mathrm{d}x+C\sum_{k}\bigg(\int_{\mathbb{R}^{n}}ug^{k}\mathrm{d}x\bigg)^{2}\nonumber 
\end{eqnarray}
We choose $\epsilon>0$ so small that $(p-2)K\epsilon\leq\lambda/2$.
Then combining with (\ref{eq:parbolic}), (\ref{eq:p order}) (\ref{eq:small})
and Sobolev-Gagliargo-Nirenberg inequality, we have
\begin{eqnarray}
 &  & \mathrm{d}\Vert u(\cdot,t)\Vert_{L_{x}^{2}}^{p}\\
 & \leq & \frac{p}{2}\Vert u\Vert_{L_{x}^{2}}^{p-2} \bigg[-\frac{\lambda}{2}\sum_{|\eta|=m}\Vert D^{\eta}u\Vert_{L_{x}^{2}}^{2}+\Vert u\Vert_{H_{x}^{m}}\Vert f\Vert_{H_{x}^{-m}}+\Vert g\Vert_{L_{x}^{2}}^{2}\nonumber \\
 &  & +\Vert g\Vert_{L_{x}^{2}}\sum_{|\eta|=m}\Vert D^{\eta}u\Vert_{L_{x}^{2}} \bigg]\mathrm{d}t+p\Vert u\Vert_{L_{x}^{2}}^{p-2}\sum_k\int_{\mathbb{R}^{n}}u\Big(\sum_{|\alpha|=m}B_{\alpha}^{k}D^{\alpha}u+g^{k}\Big)\mathrm{d}x\mathrm{d}w_{t}^{k}\nonumber \\
 & \leq & -\varepsilon\Vert u\Vert_{L_{x}^{2}}^{p-2}\Vert D^{m}u\Vert_{L_{x}^{2}}^{2}+C\Vert u\Vert_{L_{x}^{2}}^{p-2}(\Vert u\Vert_{L_{x}^{2}}^{2}+\Vert f\Vert_{H_{x}^{-m}}^{2}+\Vert g\Vert_{L_{x}^{2}}^{2})\nonumber \\
 &  & +p\Vert u\Vert_{L_{x}^{2}}^{p-2}\sum_k\int_{\mathbb{R}^{n}}u\Big(\sum_{|\alpha|=m}B_{\alpha}^{k}D^{\alpha}u+g^{k}\Big)\mathrm{d}x\mathrm{d}w_{t}^{k}.\nonumber 
\end{eqnarray}
For simplicity, we denote $\Vert D^{m}u\Vert_{L_{x}^{2}}=\sum_{|\alpha|=m}\Vert D^{\alpha}u\Vert_{L_{x}^{2}}$.
Integrating with respect to $t$ on interval $[0,s]$ for any $s\in[0,T]$,
we can obtain that
\begin{eqnarray}
 &  & \Vert u(s)\Vert_{L_{x}^{2}}^{p}+\varepsilon\int_{0}^{s}\Vert u\Vert_{L_{x}^{2}}^{p-2}\Vert D^{m}u\Vert_{L_{x}^{2}}^{2}\mathrm{d}t\label{eq:11}\\
 & \leq & C\int_{0}^{s}\Vert u\Vert_{L_{x}^{2}}^{p-2}(\Vert u\Vert_{L_{x}^{2}}^{2}+\Vert f\Vert_{H_{x}^{-m}}^{2}+\Vert g\Vert_{L_{x}^{2}}^{2})\mathrm{d}t\nonumber \\
 &  & +p\sum_k \int_{0}^{s}\Vert u\Vert_{L_{x}^{2}}^{p-2}\int_{\mathbb{R}^{n}}u(\sum_{|\alpha|=m}B_{\alpha}^{k}D^{\alpha}u+g^{k})\mathrm{d}x\mathrm{d}w_{t}^{k},\quad a.s.\nonumber 
\end{eqnarray}
where $\varepsilon=\varepsilon(m,\lambda)>0$.
Choosing the stopping time $\tau$ as before and taking expectation
on both sides of (\ref{eq:11}) and by Gronwall's inequality, one
can derive 
\begin{eqnarray}
 &  & \sup_{t\in[0,T]}\mathbb{E\Vert}u(t\wedge\tau)\Vert_{L_{x}^{2}}^{p}+\mathbb{E}\int_{0}^{\tau}\Vert u\Vert_{L_{x}^{2}}^{p-2}\Vert D^{m}u\Vert_{L_{x}^{2}}^{2}\mathrm{d}t\label{eq:12}\\
 & \leq & C\mathbb{E}\int_{0}^{\tau}\Vert u\Vert_{L_{x}^{2}}^{p-2}(\Vert f\Vert_{H_{x}^{-m}}^{2}+\Vert g\Vert_{L_{x}^{2}}^{2})\mathrm{d}t\nonumber 
\end{eqnarray}
Then we can estimate $\mathbb{E}{\sup_{t\in[0,\tau]}\Vert u(t)\Vert_{L_{x}^{2}}^{p}}$
from (\ref{eq:11}) by the BDG inequality
\begin{eqnarray}
 &  & \mathbb{E}\sup_{t\in[0,\tau]}\Vert u(t)\Vert_{L_{x}^{2}}^{p}+\mathbb{E}\int_{0}^{\tau}\Vert u\Vert_{L_{x}^{2}}^{p-2}\Vert D^{m}u\Vert_{L_{x}^{2}}^{2}\mathrm{d}t\label{eq:13}\\
 & \leq & C\mathbb{E}\int_{0}^{\tau}\Vert u\Vert_{L_{x}^{2}}^{p-2}(\Vert u\Vert_{L_{x}^{2}}^{2}+\Vert f\Vert_{H_{x}^{-m}}^{2}+\Vert g\Vert_{L_{x}^{2}}^{2})\mathrm{d}t\nonumber \\
 &  & +C\mathbb{E}\bigg\{ \int_{0}^{\tau}\Vert u\Vert_{L_{x}^{2}}^{2(p-2)}\sum_{k}\bigg[\int_{\mathbb{R}^{n}}u(\sum_{|\alpha|=m}B_{\alpha}^{k}D^{\alpha}u+g^{k})\mathrm{d}x\bigg]^{2}\mathrm{d}t\bigg\} ^{\frac{1}{2}}.\nonumber 
\end{eqnarray}
The last term of the above inequality is dominated by
\begin{eqnarray*}
 &  & C\mathbb{E}\bigg[\int_{0}^{\tau}\Vert u\Vert_{L_{x}^{2}}^{2(p-2)}(\Vert u\Vert_{L_{x}^{2}}^{2}\Vert D^{m}u\Vert_{L_{x}^{2}}^{2}+\Vert u\Vert_{L_{x}^{2}}^{2}\Vert g\Vert_{L_{x}^{2}}^{2})\mathrm{d}t\bigg]^{\frac{1}{2}}\\
 & \leq & C\mathbb{E}\bigg\{ \sup_{t\in[0,\tau]}\Vert u\Vert_{L_{x}^{2}}^{\frac{p}{2}}\bigg[\int_{0}^{\tau}(\Vert u\Vert_{L_{x}^{2}}^{p-2}\Vert D^{m}u\Vert_{L_{x}^{2}}^{2}+\Vert u\Vert_{L_{x}^{2}}^{p-2}\Vert g\Vert_{L_{x}^{2}}^{2})\mathrm{d}t\bigg]^{\frac{1}{2}}\bigg\} \\
 & \leq & \frac{1}{2}\mathbb{E}{\displaystyle \sup_{t\in[0,\tau]}\Vert u(t)\Vert_{L_{x}^{2}}^{p}}+C\mathbb{E}\int_{0}^{\tau}(\Vert u\Vert_{L_{x}^{2}}^{p-2}\Vert D^{m}u\Vert_{L_{x}^{2}}^{2}+\Vert u\Vert_{L_{x}^{2}}^{p-2}\Vert g\Vert_{L_{x}^{2}}^{2})\mathrm{d}t,
\end{eqnarray*}
which along with (\ref{eq:12}) and (\ref{eq:13}) yields that
\begin{eqnarray*}
\mathbb{E}{\sup_{t\in[0,\tau]}\Vert u(t)\Vert_{L_{x}^{2}}^{p}} & \leq & C\mathbb{E}\int_{0}^{\tau}\Vert u\Vert_{L_{x}^{2}}^{p-2}(\Vert f\Vert_{H_{x}^{-m}}^{2}+\Vert g\Vert_{L_{x}^{2}}^{2})\mathrm{d}t\\
 & \leq & \frac{1}{2}\mathbb{E}{\displaystyle \sup_{t\in[0,\tau]}\Vert u(t)\Vert_{L_{x}^{2}}^{p}}+C\mathbb{E}\bigg[\int_{0}^{\tau}(\Vert f\Vert_{H_{x}^{-m}}^{2}+\Vert g\Vert_{L_{x}^{2}}^{2})\mathrm{d}t\bigg]^{\frac{p}{2}}.
\end{eqnarray*}
Thus we obtain the estimate
\begin{equation}
\mathbb{E}{\displaystyle \sup_{t\in[0,\tau]}\Vert u(t)\Vert_{L_{x}^{2}}^{p}}\leq C\mathbb{E}\bigg[\int_{0}^{\tau}(\Vert f\Vert_{H_{x}^{-m}}^{2}+\Vert g\Vert_{L_{x}^{2}}^{2})\mathrm{d}t\bigg]^{\frac{p}{2}}.\label{eq:lem-result1}
\end{equation}
Next we need to estimate $\mathbb{E}(\int_{0}^{\tau}\Vert D^{m}u\Vert_{L_{x}^{2}}^{2}\mathrm{d}t)^{p/2}$.
Back to (\ref{eq:2 order}) and integrating with respect to time,
one can easily get that
\begin{eqnarray*}
 &  & \Vert u(\tau)\Vert_{L_{x}^{2}}^{2}+\lambda\int_{0}^{\tau}\Vert D^{m}u\Vert_{L_{x}^{2}}^{2}\mathrm{d}t\\
 & \leq & \int_{0}^{\tau}\int_{\mathbb{R}^{n}}\Big(2uf+\sum_{k}\sum_{|\alpha|=m}B_{\alpha}^{k}D^{\alpha}ug^{k}+|g|^{2}\Big)\mathrm{d}x\mathrm{d}t\\
 &  & +\sum_k\int_{0}^{\tau}\int_{\mathbb{R}^{n}}2u\Big(\sum_{|\alpha|=m}B_{\alpha}^{k}D^{\alpha}u+g^{k}\Big)\mathrm{d}x\mathrm{d}w_{t}^{k}.
\end{eqnarray*}
Computing $\mathbb{E}[\cdot]^{p/2}$ on both sides of the
above inequality and by the H\"{o}lder's inequality and BDG inequality,
we derive that
\begin{eqnarray*}
 &  & \mathbb{E}\left(\int_{0}^{\tau}\Vert D^{m}u\Vert_{L_{x}^{2}}^{2}\mathrm{d}t\right)^{\frac{p}{2}}\\
 & \leq & \epsilon\mathbb{E}\left(\int_{0}^{\tau}\Vert u\Vert_{H_{x}^{m}}^{2}\mathrm{d}t\right)^{\frac{p}{2}}+C\mathbb{E}\left[\int_{0}^{\tau}(\Vert f\Vert_{H_{x}^{-m}}^{2}+\Vert g\Vert_{L_{x}^{2}}^{2})\mathrm{d}t\right]^{\frac{p}{2}}\\
 &  & +C\mathbb{E}\bigg\{ \sum_{k}\int_{0}^{\tau}\bigg[\int_{\mathbb{R}^{n}}2u(\sum_{|\alpha|=m}B_{\alpha}^{k}D^{\alpha}u+g^{k})\mathrm{d}x\bigg]^{2}\mathrm{d}t\bigg\} ^{\frac{p}{4}}\\
 & \leq & \epsilon\mathbb{E}\left(\int_{0}^{\tau}\Vert u\Vert_{H_{x}^{m}}^{2}\mathrm{d}t\right)^{\frac{p}{2}}+C\mathbb{E}\left[\int_{0}^{\tau}(\Vert f\Vert_{H_{x}^{-m}}^{2}+\Vert g\Vert_{L_{x}^{2}}^{2})\mathrm{d}t\right]^{\frac{p}{2}}\\
 &  & +C\mathbb{E}\left[\int_{0}^{\tau}\Vert u\Vert_{L_{x}^{2}}^{2}(\Vert D^{m}u\Vert_{L_{x}^{2}}^{2}+\Vert g\Vert_{L_{x}^{2}}^{2})\mathrm{d}t\right]^{\frac{p}{4}}\\
 & \leq & 2\epsilon\mathbb{E}\left(\int_{0}^{\tau}\Vert D^{m}u\Vert_{L_{x}^{2}}^{2}\mathrm{d}t\right)^{\frac{p}{2}}+C\mathbb{E}\left[\int_{0}^{\tau}(\Vert f\Vert_{H_{x}^{-m}}^{2}+\Vert g\Vert_{L_{x}^{2}}^{2})\mathrm{d}t\right]^{\frac{p}{2}}\\
 &  & +C\mathbb{E}{\displaystyle \sup_{t\in[0,\tau]}\Vert u(t)\Vert_{L_{x}^{2}}^{p}}
\end{eqnarray*}
which along with (\ref{eq:lem-result1}) implies
the estimate \eqref{eq:001} in this case.
Here the constant $C$ further depends on $K$.
\smallskip

Finally, we replace $\tau$ in \eqref{eq:001} by the following
sequence of stopping times
\[
\tau_{k}\coloneqq\inf\bigg\{ s\geq0:\sup_{t\in[0,s]}\Vert u(s)\Vert_{L_{x}^{2}}^{2}+\int_{0}^{s}\Vert D^{m}u(s)\Vert_{L_{x}^{2}}^{2}\mathrm{d}s>k\bigg\} \wedge T,
\]
and send $k$ to infinity. 
Then \eqref{eq:001} yields the desired estimate for $l=0$ and
the lemma is proved.
\qed
\end{proof}
\begin{proposition}
Let l be a positive integer, $l\geq m$, $p\geq2,$ $r\in(0,1]$ and
$\theta\in(0,1)$. Let $u\in L_{\omega}^{p}L_{t}^{2}H_{x}^{l+m}(Q_{r})$
solve (\ref{eq:model}) in $Q_{r}$ with free terms $f\in L_{\omega}^{p}L_{t}^{2}H_{x}^{l-m}(Q_{r})$
and $g\in L_{\omega}^{p}L_{t}^{2}H_{x}^{l}(Q_{r})$. Then there exists
a constant $C$ depending only on $n$, $p$, $l$, $m$, $\lambda$, $K$, and $\theta$
such that
\begin{align}
 & \Vert D^{l}u\Vert_{L_{\omega}^{p}L_{t}^{\infty}L_{x}^{2}(Q_{\theta r})}+\Vert D^{l+m}u\Vert_{L_{\omega}^{p}L_{t}^{2}L_{x}^{2}(Q_{\theta r})}\label{eq:rescalling}\\
\leq & C\sum_{k=0}^{m-1}r^{-m-l+k}\Vert D^{k}u\Vert_{L_{\omega}^{p}L_{t}^{2}L_{x}^{2}(Q_{r})}+C\sum_{k=0}^{l-m}r^{m+k-l}\Vert D^{k}f\Vert_{L_{\omega}^{p}L_{t}^{2}L_{x}^{2}(Q_{r})}\nonumber \\
 & +C\sum_{k=0}^{l}r^{k-l}\Vert D^{k}g\Vert_{L_{\omega}^{p}L_{t}^{2}L_{x}^{2}(Q_{r})}.\nonumber 
\end{align}
Consequently, for $2(l-|\beta|)>n$,
\begin{align}
& r^{\frac{n}{2}+|\beta|}\Vert\sup_{Q_{\theta r}}|D^{\beta}u|\Vert_{L_{\omega}^{p}} \leq C\sum_{k=0}^{m-1}r^{-m+k}\Vert D^{k}u\Vert_{L_{\omega}^{p}L_{t}^{2}L_{x}^{2}(Q_{r})}\label{eq:embedding}\\
& +C\sum_{k=0}^{l-m}r^{m+k}\Vert D^{k}f\Vert_{L_{\omega}^{p}L_{t}^{2}L_{x}^{2}(Q_{r})} + C\sum_{k=0}^{l}r^{k}\Vert D^{k}g\Vert_{L_{\omega}^{p}L_{t}^{2}L_{x}^{2}(Q_{r})}\nonumber 
\end{align}
where the constant $C$ further depends on $|\beta|$. 
\end{proposition}

\begin{proof}
By Sobolev's embedding theorem, (\ref{eq:embedding}) can be derived
directly from (\ref{eq:rescalling}). Also we can reduce the problem
for general $r>0$ to the case $r=1$ by rescaling. 
Indeed, for general $r>0$,
we can apply the obtained estimates for $r=1$ to the rescaled function
\[
v(x,t)\coloneqq u(rx,r^{2m}t),\quad\forall(x,t)\in\mathbb{R}^{n}\times[-1,+\infty)
\]
which solves the equation 
\begin{eqnarray*}
\mathrm{d}v(x,t) & = & \Big[-(-1)^{m}\sum_{|\alpha|=|\beta|=m}A_{\alpha\beta}(r^{2m}t)D^{\alpha+\beta}v(x,t)+F(x,t)\Big]\mathrm{d}t\\
 &  & +\sum_{k=1}^{\infty}\Big[\sum_{|\alpha|=m}B_{\alpha}^{k}(r^{2m}t)D^{\alpha}v(x,t)+G^{k}(x,t)\Big]\mathrm{d}W_{t}^{k}
\end{eqnarray*}
with free terms
\[
F(x,t)=r^{2m}f(rx,r^{2m}t),\quad G(x,t)=r^{m}g(rx,r^{2m}t),\quad W_{t}^{k}=r^{-m}w_{r^{2m}t}^{k}
\]
and obviously, $W^{k}$ are mutually independent Wiener processes.
So it suffices to prove (\ref{eq:rescalling}) for $r=1$. By induction,
we shall only consider the case $l=m$. 

For any $\theta\in(0,1)$, choose $m+1$ cut-off functions $\xi_{i}\in C_{0}^{\infty}(\mathbb{R}^{n+1})$ with $i=1,2,\cdots,m+1$,
satisfying i) $0\leq\xi_{i}\leq1,$ ii) $\xi_{i}=1$ in $Q_{\theta_{i}}$
and $\xi_{i}=0$ outside $Q_{\theta_{i+1}}$, where $\theta_{i}=\sqrt[i]{\theta}$.
Let $v_{i}=\xi_{i}u$ which satisfy
\begin{eqnarray}
\mathrm{d}v_{i}(x,t) & = & \Big[-(-1)^{m}\sum_{|\alpha|=|\beta|=m}A_{\alpha\beta}D^{\alpha+\beta}v_{i}+f_{i}(x,t)\Big]\mathrm{d}t\label{eq:17}\\
 &  & +\sum_{k=1}^{\infty}\Big(\sum_{|\eta|=m}B_{\eta}^{k}D^{\eta}v_{i}+g_{i}^{k}\Big)\mathrm{d}w_{t}^{k},\quad\quad i=1,2,\ldots,m+1\nonumber 
\end{eqnarray}
where
\begin{eqnarray*}
f_{i} & = & \xi_{i}f+u\partial_{t}\xi_{i}-\sum_{|\alpha|=|\beta|=m}\sum_{\gamma+\eta=\alpha+\beta,|\eta|>0}C_{\gamma\eta}A_{\alpha\beta}D^{\gamma}uD^{\eta}\xi_{i}\\
g_{i}^{k} & = & \xi_{i}g^{k}-\sum_{|\alpha|=m}\sum_{\gamma+\eta=\alpha,|\eta|>0}C_{\gamma\eta}B_{\alpha}^{k}D^{\gamma}uD^{\eta}\xi_{i}
\end{eqnarray*}
where $C_{\gamma\eta}$ are the constants that can be derived from the Leibniz formula,
depending only on $\gamma$,
$\eta$, and $m$.

Applying Lemma \ref{lem:the whole space estimate} to (\ref{eq:17})
for $|\beta|=m+1-i$ with $i=1,2,\cdots,m+1,$ we have
\begin{eqnarray*}
 &  & \Vert D^{m+1-i}u\Vert_{L_{\omega}^{p}L_{t}^{\infty}L_{x}^{2}(Q_{\theta_{i}})}+\Vert D^{2m+1-i}u\Vert_{L_{\omega}^{p}L_{t}^{2}L_{x}^{2}(Q_{\theta_{i}})}\\
 & \leq & C\bigg[\sum_{k=0}^{m+1-i}\Vert D^{k}f\Vert_{L_{\omega}^{p}L_{t}^{2}H_{x}^{-m}(Q_{\theta_{i+1}})}+\sum_{k=0}^{m+1-i}\Vert D^{k}g\Vert_{L_{\omega}^{p}L_{t}^{2}L_{x}^{2}(Q_{\theta_{i+1}})}+\Vert u\Vert_{L_{\omega}^{p}L_{t}^{2}H_{x}^{2m-i}(Q_{\theta_{i+1}})}\bigg]\\
 & \leq & C\bigg[\Vert f\Vert_{L_{\omega}^{p}L_{t}^{2}L_{x}^{2}(Q_{\theta_{i+1}})}+\sum_{k=0}^{m+1-i}\Vert D^{k}g\Vert_{L_{\omega}^{p}L_{t}^{2}L_{x}^{2}(Q_{\theta_{i+1}})}+\Vert u\Vert_{L_{\omega}^{p}L_{t}^{2}H_{x}^{2m-i}(Q_{\theta_{i+1}})}\bigg].
\end{eqnarray*}
From the above inequalities, one can prove (\ref{eq:rescalling})
for $l=m$. Higher-order estimates follow from induction. The proof
is complete.
\qed
\end{proof}

Next we shall give an estimate for equation (\ref{eq:model}) with
the Dirichlet boundary conditions
\begin{equation}
\begin{cases}
u(x,0)=0, & \forall x\in\mathbb{R}^{n}\\
D^{\alpha}u(x,t)|_{\partial B_{r}}=0, & |\alpha|\leq m-1.
\end{cases}\label{eq:Dirichlet condition}
\end{equation}

\begin{proposition}
\label{prop:Diri}Let $f$ and $g$ be in $L_{\omega}^{p}L_{t}^{2}H_{x}^{l}(B_{r}\times(0,r^{2m}))$
for all $l\in\mathbb{N}$. Then the Dirichlet problem (\ref{eq:model})
with (\ref{eq:Dirichlet condition}) admits a unique weak solution
$u\in L_{\omega}^{2}L_{t}^{2}H_{x}^{m}(B_{r}\times(0,r^{2m}))$. For
each $t\in(0,r^{2m})$, $u(\cdot,t)\in L^{p}(\Omega;C^{l}(B_{\varepsilon}))$
for all $l\geq0$ and $\varepsilon\in(0,r)$. Moreover, there is a
constant $C=C(n,p,\lambda,m,K)$ such that
\begin{eqnarray}
 &  & \sum_{i=0}^{m-1}r^{i}\Vert D^{i}u\Vert_{L_{\omega}^{p}L_{t}^{2}L_{x}^{2}(B_{r}\times(0,r^{2m}))}\label{eq:Diric esti}\\
 & \leq & C\left[r^{2m}\Vert f\Vert_{L_{\omega}^{p}L_{t}^{2}L_{x}^{2}(B_{r}\times(0,r^{2m}))}+\sum_{i=0}^{m-1}r^{m+i}\Vert D^{i}g\Vert_{L_{\omega}^{p}L_{t}^{2}L_{x}^{2}(B_{r}\times(0,r^{2m}))}\right]\nonumber 
\end{eqnarray}
\end{proposition}

\begin{proof}
The existence and uniqueness of the weak solution of the Dirichlet
problem (\ref{eq:model}) and (\ref{eq:Dirichlet condition}) follow
from \cite[Section~3.2]{krylov1979stochastic}. Then we choose a cut-off function
$\varphi\in C_{0}^{\infty}(\mathbb{R}^{n})$ such that $\varphi(x)=1$
if $x\in B_{\varepsilon}$ and $\varphi(x)=0$ if $|x|>(r+\varepsilon)/2$
where $\varepsilon\in(0,r)$. Applying lemma \ref{lem:the whole space estimate}
to $v\coloneqq\varphi D^{\alpha}u$ with Sobolev's embedding theorem,
the interior regularity can be obtained. We omit the proof of the
estimate (\ref{eq:Diric esti}) because it's analogous to the proof
of (\ref{eq:p estimate}) with the help of rescaling and Sobolev-Gagliargo-Nirenberg
inequality.
\qed
\end{proof}

\section{Interior H\"{o}lder estimates for the model equation}

In this section we assume that $f\in C_{x}^{0}(\mathbb{R}^{n}\times[-1,\infty);L_{\omega}^{p})$
and $g\in C_{x}^{m}(\mathbb{R}^{n}\times[-1,\infty);L_{\omega}^{p})$,
and $f(x,t)$ and $D^{m}g(x,t)$ are Dini continuous with respect to $x$
uniformly in $t$, namely, the modulus of continuity defined by
\[
\varpi(r)=\sup_{t\geq-1,|x-y|\leq r}(\Vert f(x,t)-f(y,t)\Vert_{L_{\omega}^{p}}+\Vert D^{m}g(x,t)-D^{m}g(y,t)\Vert_{L_{\omega}^{p}})
\]
satisfies that
\[
\int_{0}^{1}\frac{\varpi(r)}{r}\mathrm{d}r<\infty
\]

\begin{theorem}\label{thm:dini}
Let u be a quasi-classical solution to (\ref{eq:model}) in $Q_{1}$.
Under the above settings, there is a constant C depending only on
n, $\lambda$, p, m and K, such that for any $X,\:Y\in Q_{1/16}$,
\begin{equation}
\Vert D^{2m}u(X)-D^{2m}u(Y)\Vert_{L_{\omega}^{p}}\leq C\left[\dis M_{1}+\int_{0}^{\dis}\frac{\varpi(r)}{r}\mathrm{d}r+\dis\int_{\dis}^{1}\frac{\varpi(r)}{r^{2}}\mathrm{d}r\right]\label{eq:Dini es}
\end{equation}
where $\varrho=|X-Y|_{\textup{p}}$ and $M_{1}=\interleave u\interleave_{m-1;Q_{1}}+\interleave f\interleave_{0;Q_{1}}+\interleave g\interleave_{m;Q_{1}}$.
\end{theorem}

\begin{proof}
Firstly we mollify the functions $u$, $f$ and $g$ in the spatial
variables. We choose a nonnegative and symmetric mollifier $\varphi:\,\mathbb{R}^{n}\rightarrow\mathbb{R}$
and define $\varphi^{\varepsilon}(x)=\varepsilon^{n}\varphi(x/\varepsilon)$,
$u^{\varepsilon}=u\ast\varphi^{\varepsilon}$, $f^{\varepsilon}=f\ast\varphi^{\varepsilon}$
and $g^{\varepsilon}=g\ast\varphi^{\varepsilon}$. It is easy to check
that $f^{\varepsilon}$ and $D^{m}g^{\varepsilon}$ are Dini continuous
and have the same modulus of continuity $\varpi$ with $f$ and $D^{m}g$
and satisfy
\begin{align*}
\interleave f^{\varepsilon}-f\interleave_{0;\mathbb{R}^{n}}+\interleave g^{\varepsilon}-g\interleave_{m;\mathbb{R}^{n}} & \rightarrow0\\
\left\Vert D^{2m}u^{\varepsilon}(X)-D^{2m}u(X)\right\Vert _{L_{\omega}^{p}}\rightarrow0\hspace{3 mm} & \forall X\in\mathbb{R}^{n}\times\mathbb{R},
\end{align*}
as $\varepsilon\rightarrow0$. On the other hand, from Fubini's theorem
one can check that $u^{\varepsilon}$ satisfies the model equation
(\ref{eq:model}) in the classical sense with free terms $f^{\varepsilon}$and
$g^{\varepsilon}$. Therefore it suffices to prove the theorem for
the mollified functions, and the general case is straightforward by
passing the limits. The readers are referred to the appendix of \cite{du2019cauchy}
for more details. Then based on the smoothness of mollified functions,
we can assume that $f$ and $g$ satisfy the following additional
condition:
\[
(\textbf{A})\quad f,g\in L_{\omega}^{p}L_{t}^{2}H_{x}^{k}(Q_{R})\cap C_{x}^{k}(Q_{R};L_{\omega}^{p})\:\mbox{for all}\:k\in\mathbb{N\mbox{ and }}R>0.
\]

From the definition of $\varpi$, one can see that for any $x,y\in\mathbb{R}^{n}$
and $t\in\mathbb{R}$,
\begin{equation}
\begin{gathered}\Vert f(x,t)-f(y,t)\Vert_{L_{\omega}^{p}}+\Vert D^{m}g(x,t)-D^{m}g(y,t)\Vert_{L_{\omega}^{p}}\leq\varpi(|x-y|)\\
\bigg\Vert D^{\beta}g(y,t)-\sum_{|\alpha|\leq m-|\beta|}\frac{D^{\alpha+\beta}g(x,t)}{\alpha!}(y-x)^{\alpha}\bigg\Vert _{L_{\omega}^{p}}\leq C(m)|x-y|^{m-|\beta|}\varpi(|x-y|),\;|\beta|\leq m
\end{gathered}
\label{eq:Dini continuous}
\end{equation}

By translation we may suppose that $X=(0,0)$ and prove the theorem
for any $Y\in Q_{1/8}$. Given $Y=(y,s)\in Q_{1/8},$ and $\tilde{\kappa}\in\mathbb{N}$
such that $\dis\coloneqq|Y|_{\textup{p}}\in[\rho^{\tilde{\kappa}+2},\rho^{\tilde{\kappa}+1}).$
With $\rho=1/2$, we denote
\[
Q^{\kappa}=Q_{\rho^{\kappa}}(0,0),\quad\kappa=0,1,2,\cdots.
\]
Let us introduce the following Dirichlet problems:
\begin{eqnarray*}
\mathrm{d}u^{\kappa} & = & \Big[-(-1)^{m}\sum_{|\gamma|=2m}A_{\gamma}D^{\gamma}u^{\kappa}+f(0,t)\Big]\mathrm{d}t\\
 &  & +\sum_{k=1}^{\infty}\bigg[\sum_{|\eta|=m}B_{\eta}^{k}D^{\eta}u^{\kappa}+\sum_{|\alpha|\leq m}\frac{D^{\alpha}g^{k}(0,t)}{\alpha!}x^{\alpha}\bigg]\mathrm{d}w_{t}^{k}\quad\mbox{in }Q^{\kappa}\\
D^{\alpha}u^{\kappa} & = & D^{\alpha}u\qquad\mbox{on }\partial_{\textup{p}}Q^{\kappa},|\alpha|\leq m-1
\end{eqnarray*}
where $\partial_{\textup{p}}Q^{\kappa}$ denotes the parabolic boundary
of the cylinder $Q^{\kappa}$ for $\kappa=0,1,2,\dots$. Then the
solvability and interior regularity of each $u^{\kappa}$ can be obtained
by applying Proposition \ref{prop:Diri} to $u^{\kappa}-u$. 

We have the following decomposition
\begin{eqnarray}
 &  & \left\Vert D^{2m}u(Y)-D^{2m}u(0)\right\Vert _{L_{\omega}^{p}}\nonumber\\
 & \leq & \left\Vert D^{2m}u^{\tilde{\kappa}}(0)-D^{2m}u(0)\right\Vert _{L_{\omega}^{p}}+\left\Vert D^{2m}u^{\tilde{\kappa}}(Y)-D^{2m}u^{\tilde{\kappa}}(0)\right\Vert _{L_{\omega}^{p}}\nonumber\\
 &  & +\left\Vert D^{2m}u^{\tilde{\kappa}}(Y)-D^{2m}u(Y)\right\Vert _{L_{\omega}^{p}}\nonumber\\
 & \eqqcolon & K_{1}+K_{2}+K_{3}.\label{eq:003}
\end{eqnarray}
The next step is to estimate the three terms respectively.
We split it into three lemmas.
\begin{lemma}
\label{claim:.K1}$$K_{1}\leq C\int_{0}^{\rho^{\tilde{\kappa}-1}}\frac{\varpi(r)}{r}\mathrm{d}r.$$
\end{lemma}

\begin{proof}
Apply (\ref{eq:embedding}) to $u^{\kappa}-u^{\kappa+1}$ with $|\beta|=l,$
$r=\rho^{\kappa+1}$, $\theta=\frac{1}{2}$ to get
\[
I_{\kappa,l}\coloneqq\interleave D^{l}(u^{\kappa}-u^{\kappa+1})\interleave{}_{0;Q^{\kappa+2}}\leq C\sum_{i=0}^{m-1}\rho^{(i-l)(\kappa+1)}\bigg\Vert \fint_{Q^{\kappa+1}}|D^{i}(u^{\kappa}-u^{\kappa+1})|^{2}\mathrm{d}X\bigg\Vert _{L_{\omega}^{p/2}}^{1/2}.
\]
In what follows, we define $\fint_{Q}=\frac{1}{|Q|}\int_{Q}$,
where $|Q|$ is the Lebesgue measure of the set $Q\subset\mathbb{R}^{n+1}$. 

On the other hand, from (\ref{eq:Diric esti}) one can obtain
\[
J_{\kappa}\coloneqq\sum_{i=0}^{m-1}\rho^{i\kappa}\left\Vert \fint_{Q_{\kappa}}|D^{i}(u^{\kappa}-u)|^{2}\mathrm{d}X\right\Vert _{L_{\omega}^{p/2}}^{1/2}\leq C\rho^{2m\kappa}\varpi(\rho^{\kappa}).
\]
Combining the above we derive
\begin{equation}
I_{\kappa,l}\leq C\rho^{-l(\kappa+1)}(J_{\kappa}+J_{\kappa+1})\leq C\rho^{(2m-l)\kappa-l}\varpi(\rho^{\kappa})\label{eq:neighbour}
\end{equation}
where $C$ is independent of $\kappa$. Choose $l=2m$, then
\[
\sum_{\kappa\geq1}\interleave D^{2m}(u^{\kappa}-u^{\kappa+1})\interleave_{0;Q^{\kappa+2}}\leq C\rho^{-2m}\sum_{\kappa\geq1}\varpi(\rho^{\kappa})\leq C\int_{0}^{1}\frac{\varpi\left(r\right)}{r}\mathrm{d}r<\infty,
\]
which implies that $D^{2m}u^{\kappa}(0)$ converges in $L_{\omega}^{p}$
as $\kappa\rightarrow\infty$. 
Here $0$ is the zero vector in $\mathbb{R}^{n+1}$. Next
we shall prove that the limit is $D^{2m}u(0)$. It suffices to prove
\begin{equation}
\lim_{\kappa\rightarrow\infty}\left\Vert D^{2m}u^{\kappa}(0)-D^{2m}u(0)\right\Vert _{L_{\omega}^{2}}=0
\end{equation}
as $p\geq2$. Applying (\ref{eq:embedding}) to $u^{\kappa}-u$ with
$|\beta|=2m$, $l=n+2m$, $r=\rho^{\kappa}$, $\theta=1/2$ and $p=2$,
we have
\begin{eqnarray*}
 &  & \sup_{Q^{\kappa+1}}\left\Vert D^{2m}(u^{\kappa}-u)\right\Vert _{L_{\omega}^{2}}^{2}\\
 & \leq & C\sum_{i=0}^{m-1}\rho^{-4m\kappa+2i\kappa}\mathbb{E\fint_{Q^{\kappa}}}|D^{i}(u^{\kappa}-u)|\mathrm{d}X+C\mathbb{E}\fint_{Q^{\kappa}}\left|f(x,t)-f(0,t)\right|^{2}\mathrm{d}X\\
 &  & +C\sum_{i\leq m}\rho^{(2i-2m)\kappa}\mathbb{E}\fint_{Q^{\kappa}}\bigg\Vert D^{i}g(x,t)-\sum_{|\alpha|\leq m-i}\frac{D^{\alpha}D^{i}g(0,t)}{\alpha!}x^{\alpha}\bigg\Vert ^{2}\mathrm{d}X\\
 &  & +C\sum_{i=1}^{n+m}\rho^{2i\kappa}\mathbb{E}\fint_{Q^{\kappa}}\left(\left|D^{i}f\right|^{2}+\left\Vert D^{i+m}g\right\Vert ^{2}\right)\textup{\ensuremath{\mathrm{d}}}X\\
 & \leq & C\sum_{i=0}^{m-1}\rho^{-4m\kappa+2i\kappa}\mathbb{E\fint_{Q^{\kappa}}}|D^{i}(u^{\kappa}-u)|^{2}\mathrm{d}X+C\varpi(\rho^{\kappa})^{2}+C\sum_{i=1}^{n+m}\rho^{2i\kappa}(\left\llbracket f\right\rrbracket _{i;Q^{\kappa}}^{2}+\left\llbracket g\right\rrbracket _{i+m;Q^{\kappa}}^{2})
\end{eqnarray*}
where the last two terms tend to 0 as $\kappa\rightarrow\infty$.
From (\ref{eq:Diric esti}) and (\ref{eq:Dini continuous}) we have
\begin{eqnarray*}
 &  & \sum_{i=0}^{m-1}\rho^{-4m\kappa+2i\kappa}\mathbb{E\fint_{Q^{\kappa}}}|D^{i}(u^{\kappa}-u)|^{2}\mathrm{d}X\\
 & \leq & C\mathbb{E}\fint_{Q^{\kappa}}\bigg(|f(x,t)-f(0,t)|^{2}+\sum_{i=0}^{m-1}\rho^{(2i-2m)\kappa}\bigg\Vert D^{i}g(x,t)-\sum_{|\alpha|\leq m-i}\frac{D^{\alpha}D^{i}g(0,t)}{\alpha!}x^{\alpha}\bigg\Vert ^{2}\bigg)\mathrm{d}X\\
 & \leq & C\varpi\left(\rho^{\kappa}\right)^{2}\rightarrow0,\qquad\mbox{as }\kappa\rightarrow\infty.
\end{eqnarray*}
Therefore $D^{2m}u^{\kappa}(0)$ converges strongly to $D^{2m}u(0)$
in $L_{\omega}^{p}$. Moreover, we have
\begin{align}
K_{1}=\left\Vert D^{2m}u^{\tilde{\kappa}}(0)-D^{2m}u(0)\right\Vert _{L_{\omega}^{p}} & \leq\sum_{j\geq\tilde{\kappa}}\interleave D^{2m}(u^{j}-u^{j+1})\interleave_{0;Q^{\kappa+2}}\label{eq:20}\\
 & \leq C\int_{0}^{\rho^{\tilde{\kappa}-1}}\frac{\varpi(r)}{r}\mathrm{d}r\nonumber 
\end{align}
where $C=C(n,m,\lambda,p,K)$.
\qed
\end{proof}
\begin{lemma}
\label{claim:K2}$$K_{2}\leq C\dis M_{1}+C\dis\int_{\dis}^{1}\frac{\varpi(r)}{r^{2}}\mathrm{d}r.$$
\end{lemma}

\begin{proof}
Define
\[
h^{\iota}\coloneqq u^{\iota}-u^{\iota-1},\quad\mbox{for }\iota=1,2,\ldots,\tilde{\kappa}.
\]
Then we decompose $K_{2}$ by
\begin{eqnarray*}
K_{2} & = & \left\Vert D^{2m}u^{\tilde{\kappa}}(Y)-D^{2m}u^{\tilde{\kappa}}(0)\right\Vert _{L_{\omega}^{p}}\\
 & \leq & \left\Vert D^{2m}u^{0}(Y)-D^{2m}u^{0}(0)\right\Vert _{L_{\omega}^{p}}+\sum_{\iota=1}^{\tilde{\kappa}}\left\Vert D^{2m}h^{\iota}(Y)-D^{2m}h^{\iota}(0)\right\Vert _{L_{\omega}^{p}}.
\end{eqnarray*}
As $D^{m+1}u^{0}$ satisfies the following homogeneous equation:
\begin{equation}
\mathrm{d}(D^{m+1}u^{0})=-(-1)^{m}\sum_{|\gamma|=2m}A_{\gamma}D^{\gamma}(D^{m+1}u^{0})\,\mathrm{d}t+\sum_{k=1}^{\infty}\sum_{|\eta|=m}B_{\eta}^{k}D^{\eta}(D^{m+1}u^{0})\,\mathrm{d}w_{t}^{k}
\end{equation}
in $Q_{3/4}.$ Using (\ref{eq:embedding}) to $D^{m+1}u^{0}$, one has
\begin{eqnarray*}
 &  & \interleave D^{3m}u^{0}\interleave_{0;Q_{1/4}}+\interleave D^{4m}u^{0}\interleave_{0;Q_{1/4}}\\
 & \leq & C\sum_{i=0}^{m-1}2^{-i}\left\Vert D^{i+m+1}u^{0}\right\Vert _{L_{\omega}^{p}L_{t}^{2}L_{x}^{2}(Q_{1/2})}\\
 & \leq & C\sum_{i=0}^{m-1}2^{-i}\left[\left\Vert D^{i+m+1}(u^{0}-u)\right\Vert _{L_{\omega}^{p}L_{t}^{2}L_{x}^{2}(Q_{1/2})}+\left\Vert D^{i+m+1}u\right\Vert _{L_{\omega}^{p}L_{t}^{2}L_{x}^{2}(Q_{1/2})}\right]
\end{eqnarray*}
Applying (\ref{eq:rescalling}) to $u$, one can get
\begin{align*}
&\sum_{i=0}^{m-1}\left\Vert D^{i+m+1}u\right\Vert _{L_{\omega}^{p}L_{t}^{2}L_{x}^{2}(Q_{1/2})}\\
&\leq C\left(\sum_{i=0}^{m-1}\left\Vert D^{i}u\right\Vert _{L_{\omega}^{p}L_{t}^{2}L_{x}^{2}(Q_{1})}+\left\Vert f\right\Vert _{L_{\omega}^{p}L_{t}^{2}L_{x}^{2}(Q_{1})}+\sum_{i=0}^{m}\left\Vert D^{i}g\right\Vert _{L_{\omega}^{p}L_{t}^{2}L_{x}^{2}(Q_{1})}\right).
\end{align*}
Applying (\ref{eq:rescalling}) and (\ref{eq:Diric esti}) to $u^{0}-u$
one can obtain
\begin{eqnarray*}
 &  & \sum_{i=0}^{m-1}\left\Vert D^{i+m+1}(u^{0}-u)\right\Vert _{L_{\omega}^{p}L_{t}^{2}L_{x}^{2}(Q_{1/2})}\\
 & \leq & C\sum_{i=0}^{m-1}\left\Vert D^{i}\left(u^{0}-u\right)\right\Vert _{L_{\omega}^{p}L_{t}^{2}L_{x}^{2}(Q_{1})}+C\left\Vert f(x,t)-f(0,t)\right\Vert _{L_{\omega}^{p}L_{t}^{2}L_{x}^{2}(Q_{1})}\\
 &  & +C\sum_{i=0}^{m}\bigg\Vert D^{i}\bigg(g(x,t)-\sum_{|\alpha|\leq m-i}\frac{D^{\alpha}g(0,t)}{\alpha!}x^{\alpha}\bigg)\bigg\Vert _{L_{\omega}^{p}L_{t}^{2}L_{x}^{2}(Q_{1})}\\
 & \leq & C\left\Vert f(x,t)-f(0,t)\right\Vert _{L_{\omega}^{p}L_{t}^{2}L_{x}^{2}(Q_{1})}+C\sum_{i=0}^{m}\bigg\Vert D^{i}\bigg(g(x,t)-\sum_{|\alpha|\leq m}\frac{D^{\alpha}g(0,t)}{\alpha!}x^{\alpha}\bigg)\bigg\Vert _{L_{\omega}^{p}L_{t}^{2}L_{x}^{2}(Q_{1})}.
\end{eqnarray*}
Therefore, 
\[
\interleave D^{3m}u^{0}\interleave_{0;Q_{1/4}}+\interleave D^{4m}u^{0}\interleave_{0;Q_{1/4}}\leq CM_{1}.
\]
Hence, for $-8^{-2m}<s\leq t\leq0$ and $x\in B_{1/8}$,
\begin{eqnarray*}
 &  & \left\Vert D^{2m}u^{0}(x,t)-D^{2m}u^{0}(x,s)\right\Vert _{L_{\omega}^{p}}\\
 & = & \bigg\Vert \int_{s}^{t}-(-1)^{m}\sum_{|\gamma|=2m}A_{\gamma}D^{\gamma}(D^{2m}u^{0})\mathrm{d}t+\int_{s}^{t}\sum_{|\eta|=m}B_{\eta}^{k}D^{\eta}(D^{2m}u^{0})\mathrm{d}w_{t}^{k}\bigg\Vert _{L_{\omega}^{p}}\\
 & \leq & C\sqrt{t-s}\left(\interleave D^{3m}u^{0}\interleave_{0;Q_{1/8}}+\interleave D^{4m}u^{0}\interleave_{0;Q_{1/8}}\right)\leq C\sqrt{t-s}M_{1}\leq C(t-s)^{\frac{1}{2m}}M_{1}
\end{eqnarray*}
Analogous to the above steps we can get 
\[
\left\Vert D^{2m+1}u^{0}\right\Vert _{L_{\omega}^{p}L_{t}^{2}L_{x}^{2}(Q_{1/4})}\leq CM_{1}.
\]
Thus we get
\begin{equation}
\left\Vert D^{2m}u^{0}(X)-D^{2m}u^{0}(Y)\right\Vert _{L_{\omega}^{p}}\leq CM_{1}|X-Y|_{\textup{p}},\quad\forall X,Y\in Q_{1/8}.\label{eq:u^0 oscillation}
\end{equation}
Note that $h^{\iota}$ satisfies
\begin{equation}
\mathrm{d}h^{\iota}=-(-1)^{m}\sum_{|\gamma|=2m}A_{\gamma}D^{\gamma}h^{\iota}\mathrm{d}t+\sum_{k=1}^{\infty}\sum_{|\eta|=m}B_{\eta}^{k}D^{\eta}h^{\iota}\mathrm{d}w_{t}^{k}.
\end{equation}
in $Q^{\iota}$. By (\ref{eq:neighbour}) we have
\begin{align*}
\rho^{-m\iota}\interleave D^{3m}h^{\iota}\interleave_{0;Q^{\iota+1}}+\interleave D^{4m}h^{\iota}\interleave_{0;Q^{\iota+1}} & \leq C\rho^{-2m\iota}\varpi(\rho^{\iota-1}),\\
\interleave D^{2m+1}h^{\iota}\interleave_{0;Q^{\iota+1}} & \leq C\rho^{-\iota}\varpi(\rho^{\iota-1}).
\end{align*}
Hence for $-\rho^{2m(\tilde{\kappa}+1)}\leq t\leq0$ and $|x|\leq\rho^{\tilde{\kappa}+1}$,
\begin{eqnarray*}
 &  & \left\Vert D^{2m}h^{\iota}(x,t)-D^{2m}h^{\iota}(x,0)\right\Vert _{L_{\omega}^{p}}\\
 & \leq & C\left(\rho^{2m\tilde{\kappa}}\interleave D^{4m}h^{\iota}\interleave_{0;Q^{\iota+1}}+\rho^{m\tilde{\kappa}}\interleave D^{3m}h^{\iota}\interleave_{0;Q^{\iota+1}}\right)\\
 & \leq & C\rho^{m(\tilde{\kappa}-\iota)}\varpi(\rho^{\iota-1}),
\end{eqnarray*}
and
\[
\left\Vert D^{2m}h^{\iota}(x,0)-D^{2m}h^{\iota}(0,0)\right\Vert _{L_{\omega}^{p}}\leq C\rho^{\tilde{\kappa}-\iota}\varpi(\rho^{\iota-1}).
\]
Combining the last two estimates and (\ref{eq:u^0 oscillation}),
we can obtain
\begin{eqnarray*}
K_{2} & = & \left\Vert D^{2m}u^{\tilde{\kappa}}(Y)-D^{2m}u^{\tilde{\kappa}}(0)\right\Vert _{L_{\omega}^{p}}\\
 & \leq & \left\Vert D^{2m}u^{0}(Y)-D^{2m}u^{0}(0)\right\Vert _{L_{\omega}^{p}}+\sum_{\iota=1}^{\tilde{\kappa}}\left\Vert D^{2m}h^{\iota}(Y)-D^{2m}h^{\iota}(0)\right\Vert _{L_{\omega}^{p}}\\
 & \leq & CM_{1}\rho^{\tilde{\kappa}+1}+C\sum_{\iota=1}^{\tilde{\kappa}}\rho^{\tilde{\kappa}-\iota}\varpi(\rho^{\iota-1})\\
 & \leq & C\dis M_{1}+C\dis\int_{\dis}^{1}\frac{\varpi(r)}{r^{2}}\mathrm{d}r.
\end{eqnarray*}
The lemma is proved.
\qed
\end{proof}
\begin{lemma}
\label{claim:K3}$$K_{3}\leq C\varpi(\dis)+C\int_{0}^{8\dis}\frac{\varpi(r)}{r}\mathrm{d}r.$$
\end{lemma}

\begin{proof}
We consider the following sequence of equations
\begin{eqnarray*}
\mathrm{d}u^{Y,\kappa} & = & \bigg[-(-1)^{m}\sum_{|\gamma|=2m}A_{\gamma}D^{\gamma}u^{Y,\kappa}+f(y,t)\bigg]\mathrm{d}t\\
 &  & +\sum_{k=1}^{\infty}\bigg[\sum_{|\eta|=m}B_{\eta}^{k}D^{\eta}u^{Y,\kappa}+\sum_{|\alpha|\leq m}\frac{D^{\alpha}g^{k}(y,t)}{\alpha!}(x-y)^{\alpha}\bigg]\mathrm{d}w_{t}^{k}\quad\mbox{in }Q^{\kappa}(Y)\\
D^{\alpha}u^{Y,\kappa} & = & D^{\alpha}u\qquad\mbox{on }\partial_{\textup{p}}Q^{\kappa}(Y),\:|\alpha|\leq m-1,\mbox{ with }\kappa=0,1,\ldots,\tilde{\kappa}-1,\tilde{\kappa}+2,\ldots;
\end{eqnarray*}
the equations associated with $\tilde{\kappa}$ and $\tilde{\kappa}+1$
are replaced by the following single equation
\begin{eqnarray*}
\mathrm{d}u^{Y,\tilde{\kappa}} & = & \bigg[-(-1)^{m}\sum_{|\gamma|=2m}A_{\gamma}D^{\gamma}u^{Y,\tilde{\kappa}}+f(y,t)\bigg]\mathrm{d}t\\
 &  & +\sum_{k=1}^{\infty}\bigg[\sum_{|\eta|=m}B_{\eta}^{k}D^{\eta}u^{Y,\tilde{\kappa}}+\sum_{|\alpha|\leq m}\frac{D^{\alpha}g^{k}(y,t)}{\alpha!}(x-y)^{\alpha}\bigg]\mathrm{d}w_{t}^{k}\quad\mbox{in }Q^{\kappa}\\
D^{\alpha}u^{Y,\tilde{\kappa}} & = & D^{\alpha}u\qquad\mbox{on }\partial_{\textup{p}}Q^{\kappa}(0),\quad|\alpha|\leq m-1.
\end{eqnarray*}
As $|Y|_{\textup{p}}\in[\rho^{\tilde{\kappa}+2},\rho^{\tilde{\kappa}+1})$,
it is easily seen that $Q^{\tilde{\kappa}+2}(Y)\subset Q^{\tilde{\kappa}}(0)$.
So analogous to the proof of (\ref{eq:20}) we have
\[
\left\Vert D^{2m}u^{Y,\tilde{\kappa}}(Y)-D^{2m}u(Y)\right\Vert _{L_{\omega}^{p}}\leq C\int_{0}^{\rho^{\tilde{\kappa}-1}}\frac{\varpi(r)}{r}\mathrm{d}r,
\]
where $C=C(n,m,\lambda,K,p)$. On the other hand, combining (\ref{eq:embedding}),
(\ref{eq:Diric esti}) and (\ref{eq:Dini continuous}), one can derive
\[
\left\Vert D^{2m}u^{Y,\tilde{\kappa}}(Y)-D^{2m}u^{\tilde{\kappa}}(Y)\right\Vert _{L_{\omega}^{p}}\leq C\varpi(\dis).
\]
Then we have 
\[
\left\Vert D^{2m}u^{\tilde{\kappa}}(Y)-D^{2m}u(Y)\right\Vert _{L_{\omega}^{p}}\leq C\varpi(\dis)+C\int_{0}^{8\dis}\frac{\varpi(r)}{r}\mathrm{d}r.
\]
The lemma is proved.
\qed
\end{proof}

Now recalling \eqref{eq:003} and combining Lemmas \ref{claim:.K1}, \ref{claim:K2} and
\ref{claim:K3}, one has that
\[
\Vert D^{2m}u(Y)-D^{2m}u(0)\Vert_{L_{\omega}^{p}}\leq C\left[\dis M_{1}+\int_{0}^{\dis}\frac{\varpi(r)}{r}\mathrm{d}r+\dis\int_{\dis}^{1}\frac{\varpi(r)}{r^{2}}\mathrm{d}r\right].
\]
The proof of Theorem \ref{thm:dini} is complete.
\qed
\end{proof}
From the above theorem, one can easily derive the following interior
H\"older estimate for (\ref{eq:model}), where we denote $\mathcal{Q}_{r,T}=B_{r}\times[0,T]$
for $r$, $T>0$.
\begin{corollary}
\label{cor:local space global time}If u is a quasi-classical solution
of (\ref{eq:model}) in $\mathbb{R}^{n}\times[0,\infty)$ with zero
initial condition and $\hd\in(0,1)$. Then there is a positive constant
C depending only on n, m, p, K, $\lambda$ and $\hd$, such that
\begin{equation}
\left\llbracket D^{2m}u\right\rrbracket _{(\hd,\hd/2m);\mathcal{Q}_{1/8,T}}\leq C\left[\interleave u\interleave_{m-1;\mathcal{Q}_{1,T}}+\frac{\interleave f\interleave_{\hd;\mathcal{Q}_{1,T}}+\interleave g\interleave_{m+\hd;\mathcal{Q}_{1,T}}}{\hd(1-\hd)}\right]
\end{equation}
for any $T>0$, provided the right-hand side is finite.
\end{corollary}

\begin{proof}
Because of the zero initial condition, define $\tilde{u}(x,t)$, $\tilde{f}(x,t)$
and $\tilde{g}(x,t)$ to be zero whenever $t\in[-1,0)$, and be equal
to $u(x,t)$, $f(x,t)$ and $g(x,t)$, respectively, whenever $t\geq0$.
Obviously, $\tilde{u}$ is a quasi-classical solution to (\ref{eq:model})
in $\mathbb{R}^{n}\times[-1,\infty)$. From (\ref{eq:Dini es}) we
have 
\[
\left\llbracket D^{2m}\tilde{u}\right\rrbracket _{\left(\hd,\hd/2m\right);Q_{1/8}(X)}\leq C\left[\interleave\tilde{u}\text{\ensuremath{\interleave}}_{m-1;Q_{1}(X)}+\frac{\interleave\tilde{f}\interleave_{\hd;Q_{1}(X)}+\interleave\tilde{g}\interleave_{m+\hd;Q_{1}(X)}}{\hd(1-\hd)}\right]
\]
for any $X=\left(x,t\right)\in\mathbb{R}^{n}\times[0,\infty)$. Using
the localization property of H\"{o}lder norms (see Lemma 4.1.1 in
\cite{krylov1996l_p}), we obtain 
\begin{eqnarray*}
 &  & \left\llbracket D^{2m}u\right\rrbracket _{(\hd,\hd/2m);\mathcal{Q}_{1/8,T}}
 \le  \left\llbracket D^{2m}\tilde{u}\right\rrbracket _{(\hd,\hd/2m);\mathcal{Q}_{1/8,T}}\\
 & \leq & C\sup_{t\in[0,T]}\left(\left\llbracket D^{2m}\tilde{u}\right\rrbracket _{(\hd,\hd/2m);Q_{1/8}(0,t)}+\interleave\tilde{u}\interleave_{0;Q_{1/8}(0,t)}\right)\\
 & \leq & C\left[\interleave u\interleave_{m-1;\mathcal{Q}_{1,T}}+\frac{\interleave\tilde{f}\interleave_{\hd;Q_{1}(X)}+\interleave\tilde{g}\interleave_{m+\hd;Q_{1}(X)}}{\hd(1-\hd)}\right].
\end{eqnarray*}
The proof is complete.
\qed
\end{proof}

\section{Global H\"{o}lder estimates and the solvability}

This section is devoted to the proof of Theorem \ref{thm:2}. 
We need two technical lemmas; readers are referred to \cite{du2019cauchy}
for their proofs.
\begin{lemma}
\label{lem:5.1}Let $\varphi$ be a bounded nonnegative function from
$[0,T]$ to $[0,\infty)$ satisfying
\[
\varphi(t)\leq\theta\varphi(s)+\sum_{i=1}^{k}C_{i}(s-t)^{-\theta_{i}},\quad\forall0\leq t<s\leq T,
\]
for some nonnegative constants $\theta,\theta_{i}$ and $C_{i}$ (i=1,$\ldots$
,k), where $\theta<1$. Then 
\[
\varphi(0)\leq C\sum_{i=1}^{k}C_{i}T^{-\theta_{i}},
\]
where $C$ depends only on $\theta_{1},\ldots\,,\theta_{k}$ and $\theta$.
\end{lemma}

\begin{lemma}
\label{lem:5.2}Let $B_{R}=\left\{ x\in\mathbb{R}^{n}:|x|<R\right\} $
with $R>0$, $p\geq1$, and $0\leq s<r$. There exists a positive
constant $C$, depending only on $n$ and $p$, such that
\[
\left\llbracket u\right\rrbracket _{s;B_{R}}\leq C\varepsilon^{r-s}\left\llbracket u\right\rrbracket _{r;B_{R}}+C\varepsilon^{-s-n/p}\left\Vert u\right\Vert _{L^{p}(B_{R};L_{\omega}^{p})}
\]
for any $u\in C^{r}(B_{R};L_{\omega}^{p})$ and $\varepsilon\in(0,R)$.
\end{lemma}

Now we are in a position to complete the proof of Theorem \ref{thm:2}.

\begin{proof}[Proof of Theorem \ref{thm:2}] The proof is divided into two steps.
\smallskip

\textbf{Step 1}. Global H\"{o}lder estimate (\ref{eq:global estimate}).

Suppose $u$ is the quasi-classical solution to (\ref{eq:dq}) with
zero initial condition. Let $\rho/2\leq r<R\leq\rho$ with $\rho\in(0,1/8)$
to be determined. Choose a nonnegative function $\zeta\in C_{0}^{\infty}(\mathbb{R}^{n})$
such that $\zeta(x)=1$ on $B_{r}$, $\zeta(x)=0$ outside $B_{R}$,
and for $\delta>0$, 
\[
\left[\zeta\right]_{\delta;\mathbb{R}^{n}}\leq C(R-r)^{-\delta}.
\]
Set $v=\zeta u$, and $A_{\alpha\beta,0}(t)=A_{\alpha\beta}(0,t)$,
$B_{\alpha,0}^{k}(t)=B_{\alpha}^{k}(0,t)$. Then $v$ satisfies 
\begin{align*}
\mathrm{d}v= & \bigg(-(-1)^{m}\sum_{|\alpha|=|\beta|=m}A_{\alpha\beta,0}D^{\alpha+\beta}v+\tilde{f}\bigg)\mathrm{d}t+\sum_{k=1}^{\infty}\bigg(\sum_{|\alpha|=m}B_{\alpha,0}^{k}D^{\alpha}u+\tilde{g^{k}}\bigg)\mathrm{d}w_{t}^{k}
\end{align*}
where 
\begin{eqnarray*}
\tilde{f} & = & (-1)^{m}\sum_{|\alpha|=|\beta|=m}A_{\alpha\beta,0}D^{\alpha+\beta}(\zeta u)-(-1)^{m}\sum_{|\alpha|,|\beta|\leq2m}\zeta A_{\alpha\beta}D^{\alpha+\beta}u+\zeta f,\\
\tilde{g^{k}} & = & -\sum_{|\alpha|=m}B_{\alpha,0}^{k}D^{\alpha}(\zeta u)+\sum_{|\alpha|\leq m}\zeta B_{\alpha}^{k}D^{\alpha}u+\zeta g^{k}.
\end{eqnarray*}
We denote $\mathcal{Q}_{R,\tau}=B_{R}\times(0,\tau)$ for $\tau>0$
and define
\[
M_{x,r}^{\tau}(u)=\sup_{t\in[0,\tau]}\bigg(\fint_{B_{r}(x)}\mathbb{E}|u(y,t)|^{p}\mathrm{d}y\bigg)^{1/p},\quad M_{r}^{\tau}(u)=\sup_{x\in\mathbb{R}^{n}}M_{x,r}^{\tau}(u).
\]
Then following from Lemma \ref{lem:5.2}, we directly derive
\begin{eqnarray*}
\interleave\tilde{f}\interleave_{\hd;\mathcal{Q}_{R,\tau}} & \leq & \left(\varepsilon+KR^{\hd}\right)\left\llbracket u\right\rrbracket _{2m+\hd;\mathcal{Q}_{R,\tau}}+C_{1}\left(R-r\right)^{-2m-\hd-n/p}M_{0,R}^{\tau}(u)\\
 &  & +\left\llbracket f\right\rrbracket _{\hd;\mathcal{Q}_{R,\tau}}+C_{1}\left(R-r\right)^{-\hd}\interleave f\interleave_{0;\mathcal{Q}_{R,\tau}},\\
\interleave\tilde{g}\interleave_{m+\hd;\mathcal{Q}_{R,\tau}} & \leq & \left(\varepsilon+KR^{\hd}\right)\left\llbracket u\right\rrbracket _{2m+\hd;\mathcal{Q}_{R,\tau}}+C_{1}\left(R-r\right)^{-2m-\hd-n/p}M_{0,R}^{\tau}(u)\\
 &  & +\left\llbracket g\right\rrbracket _{m+\hd;\mathcal{Q}_{R,\tau}}+C_{1}\left(R-r\right)^{-m-\hd}\interleave g\interleave_{0;\mathcal{Q}_{R,\tau}}.
\end{eqnarray*}
In the above two inequalities, $C_{1}=C_{1}(n,p,K,\varepsilon,\rho)$.
Applying Corollary \ref{cor:local space global time} and taking positive
$\rho$, $\varepsilon$ so small that 
\[
\varepsilon+KR^{\hd}\leq\frac{\hd(1-\hd)}{4C}
\]
where $C$ is the constant in the corollary, then we get that
\begin{eqnarray*}
\left\llbracket u\right\rrbracket _{(2m+\hd,\hd/2m);\mathcal{Q}_{r,\tau}} & \leq & \frac{3}{4}\left\llbracket u\right\rrbracket _{2m+\hd;\mathcal{Q}_{R,\tau}}+C\left(R-r\right)^{-2m-\hd-n/p}M_{0,R}^{\tau}(u)\\
 &  & +C\left(R-r\right)^{-\hd}\interleave f\interleave_{\hd;\mathcal{Q}_{R,\tau}}+C\left(R-r\right)^{-m-\hd}\interleave g\interleave_{m+\hd;\mathcal{Q}_{R,\tau}}.
\end{eqnarray*}
Then by Lemma \ref{lem:5.1}, we obtain
\[
\left\llbracket u\right\rrbracket _{(2m+\hd,\hd/2m);\mathcal{Q}_{\rho/2,\tau}}\leq C\left(M_{0,\rho}^{\tau}(u)+\interleave f\interleave_{\hd;\mathcal{Q}_{\rho,\tau}}+\interleave g\interleave_{m+\hd;\mathcal{Q}_{\rho,\tau}}\right).
\]
Note that the above inequality is true for any point $x\in\mathbb{R}^{n}$
instead of $0$. Therefore, applying Lemma \ref{lem:5.2}, we have
\begin{eqnarray*}
\sup_{x\in\mathbb{R}^{n}}\interleave u\interleave_{(2m+\hd,\hd/2m);\mathcal{Q}_{\rho/2,\tau}(x)} & \leq & C\left(M_{\rho}^{\tau}(u)+\interleave f\interleave_{\hd;\mathcal{Q}_{\rho,\tau}}+\interleave g\interleave_{m+\hd;\mathcal{Q}_{\rho,\tau}}\right)\\
 & \leq & C\left(M_{\rho/2}^{\tau}(u)+\interleave f\interleave_{\hd;\mathcal{Q}_{\tau}}+\interleave g\interleave_{m+\hd;\mathcal{Q}_{\tau}}\right).
\end{eqnarray*}
The next step is to estimate $M_{\rho/2}^{\tau}(u)$. Applying It\^{o}'s
formula to $|u|^{p}$ and integrating in $\mathcal{Q}_{\rho/2,\tau}(x)\times\Omega$
with the use of Sobolev-Gagliargo-Nirenberg inequality, we get
\[
M_{\rho/2}^{\tau}(u)\leq C_{2}\tau\left(\sup_{x\in\mathbb{R}^{n}}\interleave u\interleave_{2m;\mathcal{Q}_{\rho/2,\tau}(x)}+\interleave f\interleave_{0;\mathcal{Q}_{\tau}}+\interleave g\interleave_{0;\mathcal{Q}_{\tau}}\right).
\]
Taking $\tau=2\left(CC_{2}\right)^{-1},$ the above two inequalities
yield
\begin{eqnarray*}
\sup_{x\in\mathbb{R}^{n}}\interleave u\interleave_{(2m+\hd,\hd/2m);\mathcal{Q}_{\rho/2,\tau}(x)} & \leq & C\left(\interleave f\interleave_{\hd;\mathcal{Q}_{\tau}}+\interleave g\interleave_{m+\hd;\mathcal{Q}_{\tau}}\right).
\end{eqnarray*}
Following from the localization property of H\"{o}lder norms, we
get
\begin{eqnarray}
\interleave u\interleave_{(2m+\hd,\hd/2m);\mathcal{Q}_{\tau}} & \leq & C_{\tau}\left(\interleave f\interleave_{\hd;\mathcal{Q}_{\tau}}+\interleave g\interleave_{m+\hd;\mathcal{Q}_{\tau}}\right)\label{eq:tauest}
\end{eqnarray}
with $C_{\tau}=C_{\tau}(n,m,\hd,\lambda,K,p)\geq1.$

Finally, we conclude the proof by induction. Assume that there is
a constant $C_{S}\geq1$ for some $S>0$ such that 
\[
\interleave u\interleave_{(2m+\hd,\hd/2m);\mathcal{Q}_{S}}\leq C_{S}\left(\interleave f\interleave_{\hd;\mathcal{Q}_{S}}+\interleave g\interleave_{m+\hd;\mathcal{Q}_{S}}\right).
\]
Then applying (\ref{eq:tauest}) to $v(x,t)\coloneqq u(x,t+S)-u(x,S)$
for $t\geq0$, we can derive that
\begin{eqnarray*}
\interleave v\interleave_{(2m+\hd,\hd/2m);\mathcal{Q}_{\tau}} & \leq & C_{\tau}\left(\interleave f\interleave_{\hd;\mathcal{Q}_{S+\tau}}+\interleave g\interleave_{m+\hd;\mathcal{Q}_{S+\tau}}+\tilde{C}\interleave u\interleave_{(2m+\hd,\hd/2m);\mathcal{Q}_{S}}\right)\\
 & \leq & C_{\tau}(1+\tilde{C}C_{S})\left(\interleave f\interleave_{\hd;\mathcal{Q}_{S+\tau}}+\interleave g\interleave_{m+\hd;\mathcal{Q}_{S+\tau}}\right)
\end{eqnarray*}
where $\tilde{C}=\tilde{C}(m,K)\geq1$. Thus we get 
\begin{eqnarray*}
\interleave u\interleave_{(2m+\hd,\hd/2m);\mathcal{Q}_{S+\tau}} & \leq & \interleave v\interleave_{(2m+\hd,\hd/2m);\mathcal{Q}_{\tau}}+\interleave u\interleave_{(2m+\hd,\hd/2m);\mathcal{Q}_{S}}\\
 & \leq & 3\tilde{C}C_{\tau}C_{S}\left(\interleave f\interleave_{\hd;\mathcal{Q}_{S+\tau}}+\interleave g\interleave_{m+\hd;\mathcal{Q}_{S+\tau}}\right)
\end{eqnarray*}
which means $C_{S+\tau}\leq3\tilde{C}C_{\tau}C_{S}$. As $\tau$ is
fixed, by iteration we have $C_{S}\leq C\textup{e}^{CS}$ where $C=C(n,m,\hd,\lambda,p,K)$.
This completes the proof of (\ref{eq:global estimate}).
\smallskip

\textbf{Step 2}. The solvability.

For simplicity, we denote
\[
L=-\left(-1\right)^{m}\sum_{\left|\alpha\right|,|\beta|\leq m}A_{\alpha\beta}D^{\alpha+\beta},\quad\varLambda^{k}=\sum_{\left|\alpha\right|\leq m}B_{\alpha}^{k}D^{\alpha}.
\]
Define
\[
L_{s}=sL+(1-s)\Delta^{2m},\quad\varLambda_{s}^{k}=s\varLambda^{k}
\]
where $s\in[0,1]$ and $\Delta^{2m}\coloneqq\sum_{\left|\gamma\right|=2m}\delta_{\gamma}D^{\gamma}$
where 
\[
\delta_{\gamma}\coloneqq\begin{cases}
1, & \gamma_{i}=2m\:\mbox{for some }1\leq i\leq n\\
0, & \mbox{other}.
\end{cases}
\]
Then consider the equation
\begin{equation}
\begin{aligned} & \mathrm{d}u=(L_{s}u+f)\mathrm{d}t+\sum_{k=1}^{\infty}(\varLambda_{s}^{k}u+g^{k})\mathrm{d}w_{t}^{k}\;\mbox{in }\mathcal{Q}\\
 & u(\cdot,0)=0\;\mbox{in }\mathbb{R}^{n}
\end{aligned}
\end{equation}
where $\mathcal{Q}\coloneqq\mathbb{R}^{n}\times[0,\infty)$. 
Evidently, the solutions of the above equations enjoy the estimate (\ref{eq:global estimate}) 
with the same dominating constant $C$ (independent of $s$).
So by the standard method of continuity (see \cite[Theorem~5.2]{gilbarg2015elliptic}), it suffices to show the solvability
of the following equation (the case $s=0$):
\begin{equation}
\mathrm{d}u=(\Delta^{2m}u+f)\mathrm{d}t+\sum_{k=1}^{\infty}g^{k}\mathrm{d}w_{t}^{k},\;u(\cdot,0)=0.\label{eq:simple eq}
\end{equation}

Letting $\text{\ensuremath{\varphi}:\ensuremath{\mathbb{R}}}^{n}\rightarrow\mathbb{R}$
be a nonnegative and symmetric mollifier (see Appendix in \cite{du2019cauchy})
and $\varphi^{\varepsilon}(x)=\varepsilon^{n}\varphi(x/\varepsilon)$,
we define $f^{\varepsilon}=\varphi^{\varepsilon}\ast f$ and $g^{\varepsilon}=\varphi^{\varepsilon}\ast g$.
From the results of Appendix in \cite{du2019cauchy}, we obtain that
$f^{\varepsilon}\in C_{x}^{\hd}(\mathcal{Q};L_{\omega}^{p})$ and
$g^{\varepsilon}\in C_{x}^{m+\hd}(\mathcal{Q};L_{\omega}^{p})$ satisfying
\begin{equation}
\interleave f^{\varepsilon}-f\interleave_{\hd/2;\mathcal{Q}_{T}}+\interleave g^{\varepsilon}-g\interleave_{m+\hd/2;\mathcal{Q}_{T}}\rightarrow0\;\mbox{as }\varepsilon\rightarrow0.\label{eq:f g estimate}
\end{equation}
Moreover, $f^{\varepsilon}(x,t,\omega)$ and $g^{\varepsilon}(x,t,\omega)$
are smooth in $x$ for any $(t,\omega)$, and $f^{\varepsilon}$,
$g^{\varepsilon}\in C^{k}(\mathcal{Q}_{T};L_{\omega}^{p})$ for all
$k\in\mathbb{N}$. For any $k\in\mathbb{N}$ and $2r>n$, we have
\begin{eqnarray*}
 &  & \mathbb{E}\left|\int_{\mathcal{Q}_{T}}(1+|x|^{2})^{-r}(|D^{k}f^{\varepsilon}(x,t)|^{2}+|D^{k}g^{\varepsilon}(x,t)|^{2})\mathrm{d}x\mathrm{d}t\right|^{\frac{p}{2}}\\
 & \leq & C\left|\int_{\mathcal{Q}_{T}}\left(1+|x|^{2}\right)^{-r}\left(\mathbb{E}|D^{k}f^{\varepsilon}|^{p}+\mathbb{E}|D^{k}g^{\varepsilon}|^{p}\right)^{\frac{2}{p}}\mathrm{d}x\mathrm{d}t\right|^{\frac{p}{2}}\\
 & \leq & CT^{\frac{p}{2}}\left(\interleave f^{\varepsilon}\interleave_{k;\mathcal{Q}_{T}}^{p}+\interleave g^{\varepsilon}\interleave_{k;\mathcal{Q}_{T}}^{p}\right)<\infty.
\end{eqnarray*}
Considering the weighted Sobolev spaces and analogously to proving
Lemma \ref{lem:the whole space estimate} but only with minor changes,
we derive that (\ref{eq:simple eq}) with free terms $f^{\varepsilon}$
and $g^{\varepsilon}$ admits a unique weak solution $u^{\varepsilon}$
satisfying 
\[
\mathbb{E}\sup_{t\in[0,T]}\left|\int_{\mathbb{R}^{n}}(1+|x|^{2})^{-r}\left|D^{k}u^{\varepsilon}(x,t)\right|^{2}\mathrm{d}x\right|^{\frac{p}{2}}<\infty\quad\forall k\in\mathbb{N}
\]
for any large $r$. Following from Sobolev's embedding theorem, $u^{\varepsilon}$
is smooth in $x$ and moreover, 
\begin{eqnarray*}
 &  & \mathbb{E}\sup_{(x,t)\in\mathcal{Q}_{T}}(1+|x|^{2})^{-\frac{rp}{2}}|D^{k}u^{\varepsilon}(x,t)|^{p}\\
 & \leq & C\mathbb{E}\sup_{t\in[0,T]}\Big\Vert (1+\left|x\right|^{2})^{-\frac{r}{2}}D^{k}u^{\varepsilon}\Big\Vert _{H_{x}^{n}}^{p}\\
 & \leq & C\mathbb{E}\sup_{t\in[0,T]}\sum_{i=0}^{n}\left|\int_{\mathbb{R}^{n}}(1+|x|^{2})^{-r}|D^{k+i}u^{\varepsilon}(x,t)|^{2}\mathrm{d}x\right|^{\frac{p}{2}}\\
 & < & \infty.
\end{eqnarray*}
Then we have $\mathbb{E}\left|D^{k}u^{\varepsilon}(x,t)\right|^{p}<\infty$
for each $(x,t)\in\mathcal{Q}_{T}$ and $k\in\mathbb{N}.$ From global
estimate (\ref{eq:global estimate}) with $\hd/2$ instead of $\hd$
and (\ref{eq:f g estimate}), we obtain
\[
\interleave u^{\varepsilon}-u^{\varepsilon'}\interleave_{2m;\mathcal{Q}_{T}}\leq C\left(\interleave f^{\varepsilon}-f^{\varepsilon'}\interleave_{\hd/2;\mathcal{Q}_{T}}+\interleave g^{\varepsilon}-g^{\varepsilon'}\interleave_{m+\hd/2;\mathcal{Q}_{T}}\right)\rightarrow0
\]
as $\varepsilon,\varepsilon'\rightarrow0$. Hence, $u^{\varepsilon}$
converges to a function $u\in C_{x,t}^{2m,0}(\mathcal{Q}_{T};L_{\omega}^{p})$
which is apparently a quasi-classical solution to (\ref{eq:simple eq}).
Then we can derive the uniqueness and regularity from the estimate
(\ref{eq:global estimate}). The solvability is proved.

To sum up, the proof of Theorem \ref{thm:2} is complete.
\qed
\end{proof}

\section{Proof of Lemma \ref{lem:sharpness}}

Recall Equation \eqref{eq:eg}:
\[
\textrm{d}u=(-1)^{m+1}D^{2m}u\textrm{d}t+\mu D^{m}u\textrm{d}w_{t}
\]
with the initial condition with the initial condition 
$$u(x,0)=\sum_{n\in\mathbb{Z}}\mathrm{e}^{-n^{2m}}\cdot\mathrm{e}^{\sqrt{-1}nx},
\quad x\in\mathbb{T}=\mathbb{R}/2\pi\mathbb{Z}.$$ 

Since $\mu^{2}<2$,
this equation admits a unique solution $u\in L^{2}(\Omega;C([0,T];H^{l}(\mathbb{T})))$
for any integer $l$ (cf.~\cite{krylov1979stochastic}). 
We shall prove that if $p>1+2/\mu^2$, then $\mathbb{E}\Vert u(\cdot,t)\Vert _{L^{2}(\mathbb{T})}^{p}=+\infty$
for any $t>2/\varepsilon$, where $\varepsilon = (p-1)\mu^{2}-2>0$.

Since $u\in L^{2}(\Omega;C([0,T];H^{l}(\mathbb{T})))$ for any
integer $l$, one can express $u$ in the Fourier series 
\[
u(x,t)=\sum_{n\in\mathbb{Z}}u_{n}(t)\mathrm{e}^{\sqrt{-1}nx}
\]
where $u_{n}(t)$ satisfies
\begin{align*}
\textrm{d}u_{n} & =(-1)^{m+1}(\sqrt{-1}n)^{2m}u_{n}\textrm{d}t+\mu(\sqrt{-1}n)^{m}u_{n}\textrm{d}w_{t}\\
 & =u_{n}(-n^{2m}\mathrm{d}t+\mu(\sqrt{-1})^{m}n^{m}\mathrm{d}w_{t}),\\
u_{n}(0) & =\mathrm{e}^{-n^{2m}}.
\end{align*}
Then we obtain that
\begin{align}\label{eq:odevity}
u_{n}(t)=\exp\left\{ -n^{2m}(1+t+\frac{(-1)^{m}\mu^{2}}{2}t)+\mu(\sqrt{-1})^{m}n^{m}w_{t}\right\} .
\end{align}
Set $f(t)\coloneqq2+2t+(-1)^{m}\mu^{2}t$, then 
\begin{align*}
|u_{n}(t)|^{2} & =\left|\exp\left\{ -n^{2m}f(t)+2\mu(\sqrt{-1})^{m}n^{m}w_{t}\right\} \right|\\
 & =\left|\exp\left\{ -f(t)\Big(n^{m}-\frac{\mu(\sqrt{-1})^{m}w_{t}}{f(t)}\Big)^{2}+\frac{\mu^{2}(-1)^{m}|w_{t}|^{2}}{f(t)}\right\} \right|.\nonumber 
\end{align*}

Using the condition that $m$ is \emph{even}, one can obtain 
\[
|u_{n}(t)|^{2}=\exp\left\{ -f(t)\left(n^{m}-\frac{\mu(-1)^{m/2}w_{t}}{f(t)}\right)^{2}+\frac{\mu^{2}|w_{t}|^{2}}{f(t)}\right\} .
\]
By Parseval's identity,
\begin{align*}
\left\Vert u(\cdot,t)\right\Vert _{L^{2}(\mathbb{T})}^{2} & =2\pi\sum_{n\in\mathbb{Z}}|u_{n}(t)|^{2}\\
 & =2\pi\sum_{n\in\mathbb{Z}}\exp\left\{ -f(t)\left(n^{m}-\frac{\mu(-1)^{m/2} w_{t}}{f(t)}\right)^{2}+\frac{\mu^{2}|w_{t}|^{2}}{f(t)}\right\} .
\end{align*}
Therefore, we have 
\begin{align*}
& \mathbb{E}\left\Vert u(\cdot,t)\right\Vert _{L^{2}(\mathbb{T})}^{p} 
=(2\pi)^{\frac{p}{2}}\mathbb{E}\left(\sum_{n\in\mathbb{Z}}\exp\left\{ -f(t)\left(n^{m}-\frac{\mu(-1)^{m/2}w_{t}}{f(t)}\right)^{2}+\frac{\mu^{2}|w_{t}|^{2}}{f(t)}\right\} \right)^{\frac{p}{2}}\\
 & =(2\pi)^{\frac{p-1}{2}}\int_{-\infty}^{+\infty}\exp\left\{ -\frac{y^{2}}{2}\right\} \left(\sum_{n\in\mathbb{Z}}\exp\left\{ -f(t)\left(n^{m}-\frac{\mu(-1)^{m/2} y}{f(t)/\sqrt{t}}\right)^{2}+\frac{\mu^{2}y^{2}}{f(t)/t}\right\} \right)^{\frac{p}{2}}\mathrm{d}y\\
 & = (2\pi)^{\frac{p-1}{2}}\int_{-\infty}^{+\infty}\exp\left\{ -\frac{y^{2}}{2}\left(1-\frac{p\mu^{2}}{f(t)/t}\right)\right\} \left(\sum_{n\in\mathbb{Z}}\exp\left\{ -f(t)\left(n^{m}-\frac{\mu(-1)^{m/2} y}{f(t)/\sqrt{t}}\right)^{2}\right\} \right)^{\frac{p}{2}}\mathrm{d}y.
\end{align*}
Noticing the fact that
$\mu(-1)^{m/2} y$ is positive on one of the intervals $(-\infty,0)$ and $(0,+\infty)$,
one has 
\begin{align*}
& \mathbb{E}\left\Vert u(\cdot,t)\right\Vert _{L^{2}(\mathbb{T})}^{p} \\
 & \geq(2\pi)^{\frac{p-1}{2}}\int_{0}^{+\infty}\exp\left\{ -\frac{y^{2}}{2}\left(1-\frac{p\mu^{2}}{f(t)/t}\right)\right\} \left(\sum_{n\in\mathbb{Z}}\exp\left\{ -f(t)\left(n^{m}-\frac{|\mu|y}{f(t)/\sqrt{t}}\right)^{2}\right\} \right)^{\frac{p}{2}}\mathrm{d}y\\
  & \geq(2\pi)^{\frac{p-1}{2}}\sum_{n\in\mathbb{Z}}\int_{0}^{+\infty} 
  \exp\left\{ -\frac{y^{2}}{2}\left(1-\frac{p\mu^{2}}{f(t)/t}\right)\right\} 
  \exp\left\{ -\frac{p}{2}f(t)\left(n^{m}-\frac{|\mu|y}{f(t)/\sqrt{t}}\right)^{2}\right\} \mathrm{d}y.
\end{align*}
Since 
\[
\Big| n^{m}-\frac{|\mu|y}{f(t)/\sqrt{t}}\Big| \le 1\quad 
\text{when } y\in \Big[{n^mf(t) \over |\mu|\sqrt{t}},{(n^m+1) f(t) \over |\mu|\sqrt{t}}\Big] ,
\]
one can further derive that
\begin{align*}
& \mathbb{E}\left\Vert u(\cdot,t)\right\Vert _{L^{2}(\mathbb{T})}^{p} \\
 & \geq(2\pi)^{\frac{p-1}{2}}\sum_{n\in\mathbb{Z}}\int_{n^{m}f(t)/(|\mu|\sqrt{t})}^{(n^{m}+1)f(t)/(|\mu|\sqrt{t})}\exp\left\{ -\frac{y^{2}}{2}\left(1-\frac{p\mu^{2}}{f(t)/t}\right)\right\} \exp\left\{ -\frac{p}{2}f(t)\right\} \mathrm{d}y\\
 & \geq(2\pi)^{\frac{p-1}{2}} \mathrm{e}^{-{p\over 2}f(t)}
 \sum_{n\in\mathbb{Z}}\int_{0}^{f(t) \over |\mu|\sqrt{t}}\exp\left\{ -\frac{(y+n^{m})^{2}}{2}\left(1-\frac{p\mu^{2}}{f(t)/t}\right)\right\} \mathrm{d}y.
\end{align*}
Obviously,
the last term is infinite if 
\begin{equation}
1-\frac{p\mu^{2}}{f(t)/t}<0,\label{eq:lambdap}
\end{equation}
which is satisfied when $t>2/\varepsilon$, where $\varepsilon = (p-1)\mu^{2}-2>0$. 
The lemma is proved.

\begin{remark}
When $m$ is odd, it follows from (\ref{eq:odevity}) that 
\[
|u_{n}(t)|^{2}=\exp\left\{ -n^{2m}f(t)\right\} 
\]
where $f(t)=2+(2-\mu^{2})t$. 
Furthermore, one can obtain
\[
\mathbb{E}\left\Vert u(\cdot,t)\right\Vert _{L^{2}(\mathbb{T})}^{p}=(2\pi)^{\frac{p}{2}}\left(\sum_{n\in\mathbb{Z}}\exp\left\{ -n^{2m}f(t)\right\} \right)^{\frac{p}{2}},
\]
which means that the condition $\mu^{2}<2$ is sufficient to ensure $\mathbb{E}\left\Vert u(\cdot,t)\right\Vert _{L^{2}(\mathbb{T})}^{p}<+\infty$ for any $p\ge 2$.
\end{remark}


\bibliographystyle{spmpsci}      
\bibliography{bibdata}   

%
%

\end{document}